\documentclass{gtart}

\pdfoutput=1

\newcommand{\pathtotrunk}{./}

\usepackage{amsmath,amssymb,amsfonts}
\usepackage{ifpdf}
\usepackage{enumerate}
\usepackage{ulem}
\usepackage{leftidx}

\newcommand{\pathtodiagrams}{\pathtotrunk}

\newcommand{\mathfig}[2]{{\hspace{-3pt}\begin{array}{c}%
  \raisebox{-2.5pt}{\includegraphics[width=#1\textwidth]{\pathtodiagrams #2}}%
\end{array}\hspace{-3pt}}}

\newcommand{\arxiv}[1]{\href{http://arxiv.org/abs/#1}{\tt arXiv:\nolinkurl{#1}}}
\newcommand{\doi}[1]{\href{http://dx.doi.org/#1}{{\tt DOI:#1}}}

\newcommand{\googlebooks}[1]{(preview at \href{http://books.google.com/books?id=#1}{google books})}

\def\RCS$#1: #2 ${\expandafter\def\csname RCS#1\endcsname{#2}}
\RCS$Revision: 1998 $ \RCS$Date: 2010-08-10 04:32:36 -0700 (Tue, 10 Aug 2010) $

\theoremstyle{plain}
\newtheorem{prop}{Proposition}[section]

\newtheorem{thm}[prop]{Theorem}

\newtheorem{lem}[prop]{Lemma}
\newtheorem{cor}[prop]{Corollary}
\newtheorem*{cor*}{Corollary}

\newtheorem{defn}[prop]{Definition}         
\newtheorem*{defn*}{Definition}             

\newenvironment{rem}{\noindent\textsl{Remark.}}{}  
\numberwithin{equation}{section}

\newcounter{comment}
\newcommand{\noop}[1]{}

\def\clap#1{\hbox to 0pt{\hss#1\hss}}

\newcommand{\Natural}{\mathbb N}

\newcommand{\Real}{\mathbb R}
\newcommand{\Complex}{\mathbb C}

\def\semicolon{;}
\def\applytolist#1{
    \expandafter\def\csname multi#1\endcsname##1{
        \def\multiack{##1}\ifx\multiack\semicolon
            \def\next{\relax}
        \else
            \csname #1\endcsname{##1}
            \def\next{\csname multi#1\endcsname}
        \fi
        \next}
    \csname multi#1\endcsname}

\def\calc#1{\expandafter\def\csname c#1\endcsname{{\mathcal #1}}}
\applytolist{calc}QWERTYUIOPLKJHGFDSAZXCVBNM;

\def\bfc#1{\expandafter\def\csname b#1\endcsname{{\mathbf #1}}}
\applytolist{bfc}QWERTYUIOPLKJHGFDSAZXCVBNM;

\renewcommand{\imath}{\mathfrak{i}}
\renewcommand{\jmath}{\mathfrak{j}}

\newcommand{\setc}[2]{\left\{#1 \;\left| \; #2 \right. \right\}}

\newcommand{\code}[1]{{\tt #1}}

\newenvironment{mma}{}{}

\newcounter{lineNumber}
\newcommand{\inN}{\stepcounter{lineNumber}\arabic{lineNumber}}
\newcommand{\outN}{\arabic{lineNumber}}

\newenvironment{inm}{%
 \vskip 0pt \par\noindent%
  {\footnotesize\color{blue}In[\inN]$:=$}\tt%
  }{}
\newenvironment{outm}{%
 \vskip -6pt \par\noindent\parindent=12pt%
  {\footnotesize\color{blue}Out[\outN]$=$}\tt%
  }{}
    {\begin{quotation}\noindent}%
    {\end{quotation}}%

\newcommand{\bigraph}[1]{{\hspace{-3pt}\begin{array}{c}%
  \raisebox{-2.5pt}{\includegraphics[height=6mm]{\pathtodiagrams #1}}%
\end{array}\hspace{-3pt}}}

\newcommand{\eset}{\emptyset}

\newcommand{\tensor}{\otimes}

\ifpdf
\usepackage[pdftex,plainpages=false,hypertexnames=false,pdfpagelabels]{hyperref}
\else
\usepackage[dvips,plainpages=false,hypertexnames=false]{hyperref}
\fi

\ifpdf
\usepackage[pdftex]{graphicx}
\usepackage[pdftex]{rotating}
\else
\usepackage[dvips]{graphicx}
\usepackage[dvips]{rotating}
\fi

\usepackage{microtype}

\usepackage{tikz}
\usepackage{tikz-qtree}
\usetikzlibrary{shapes}
\usetikzlibrary{backgrounds}

\usepackage{color}

\newtheorem{example}{Example}

\title{Subfactors of index less than 5, part 1: the principal graph odometer}

\author{Kim~Morrison}
\address{
}%
\email{scott@tqft.net}

\author{Noah~Snyder}
\address{
}%
\email{nsnyder@math.columbia.edu}

\address{%
\rm URLs:\stdspace \tt \url{http://tqft.net/}
\rm and \tt
\url{http://math.columbia.edu/~nsnyder}}


\begin{document}

\begin{abstract}In this series of papers we show that there are exactly ten subfactors, other than $A_\infty$ subfactors, of index between $4$ and $5$.  Previously this classification was known up to index $3+\sqrt{3}$.  In the first paper we give an analogue of Haagerup's initial classification of subfactors of index less than $3+\sqrt{3}$, showing that any subfactor of index less than $5$ must appear in one of a large list of families.  These families will be considered separately in the three subsequent papers in this series.
\end{abstract}

\maketitle



\section{Introduction}
A subfactor is an inclusion $N \subset M$ of von Neumann algebras with trivial center.  The theory of subfactors can be thought of as a nonabelian version of Galois theory, and has had many applications in operator algebras, quantum algebra, and knot theory.  For example, given a finite depth subfactor, you can produce two fusion categories (by taking the even parts) and a $3$-dimensional TQFT (via the Ocneanu-Turaev-Viro construction \cite{MR1191386, MR1317353, MR1686423}).  The fundamental example of a subfactor, which illustrates the analogy with Galois theory, is when $G$ is a finite group with an outer action on a factor $M$ and the fixed point factor $M^G$ is a subfactor of $M$.

A subfactor $N \subset M$ has three key invariants.  From strongest to weakest, they are: the standard invariant (which captures all information about ``basic'' bimodules over $M$ and $N$), the principal and dual principal graphs (which together describe the fusion rules for these basic bimodules), and the index (which is a real number measuring the ``size'' of the basic bimodules \cite{MR696688}).  For the case of the fixed point subfactor the standard invariant recovers the structure of $G$ itself, the principal graph and dual principal graph recover the size of the group and the dimensions of each of its irreducible representations, and the index is the size of $G$.  Thus studying possible standard invariants (called paragroups \cite{MR996454}, $\lambda$-lattices \cite{MR1334479}, or subfactor planar algebras\cite{math.QA/9909027}) of a fixed index is a generalization of studying all groups of a given size.  In turn, studying subfactors with a fixed standard invariant is a generalization of studying all outer actions of a group on a factor.  In particular, just as any finite (or more generally, amenable group) has only one action on the hyperfinite $II_1$ factor \cite{MR587749,MR596082}, any finite depth (or more generally amenable) subfactor planar algebra can be realized in a unique way as a subfactor of the hyperfinite $II_1$ factor \cite{MR1055708, MR1278111}.

Much early work in the theory of subfactors concerned the classification of subfactors of index up to $4$ \cite{MR999799, MR996454, MR1145672, MR1936496}.  One reason for concentrating on index below $4$ was that for every real number greater than $4$ there is a subfactor with that index and principal graph $A_\infty$ \cite{MR1198815}.  The study of subfactors of small index larger than $4$ was initiated by Haagerup who gave an exhaustive list of possible principal graphs other than $A_\infty$ of subfactors of index less than $3+\sqrt{3}$ \cite{MR1317352}.  Most of the graphs on this list were excluded by Bisch \cite{MR1625762} and Asaeda-Yasuda \cite{MR2472028}, while the remaining $3$ graphs and their duals were shown to come from unique subfactor planar algebras/$\lambda$-lattices/paragroups (and thus by Popa \cite{MR1055708, MR1278111} unique subfactors of the hyperfinite $II_1$ factor) by Asaeda-Haagerup \cite{MR1686551} and Bigelow-Morrison-Peters-Snyder \cite{0909.4099}.  The goal of this series of papers is to give the following classification of irreducible subfactors of index less than $5$.  

\begin{thm}
\label{conj:4-5}
There are exactly ten subfactor planar algebras other than Temperley-Lieb with index between $4$ and $5$.
\begin{itemize}
\item The Haagerup planar algebra \cite{MR1686551}, with index $\frac{5+\sqrt{13}}{2}$ and principal bigraph pair $$\left(\bigraph{bwd1v1v1v1p1v1x0p0x1v1x0p0x1duals1v1v1x2v2x1} \bigraph{bwd1v1v1v1p1v1x0p1x0duals1v1v1x2}\right)$$ and its dual.
\item The extended Haagerup planar algebra \cite{0909.4099}, with index $\frac{8}{3}+\frac{2}{3} \operatorname{Re} \sqrt[3]{\frac{13}{2} \left(-5-3 i \sqrt{3}\right)}$ and principal bigraph pair $$\left(\bigraph{bwd1v1v1v1v1v1v1v1p1v1x0p0x1v1x0p0x1duals1v1v1v1v1x2v2x1}, \bigraph{bwd1v1v1v1v1v1v1v1p1v1x0p1x0duals1v1v1v1v1x2}\right)$$ and its dual.
\item The Asaeda-Haagerup planar algebra \cite{MR1686551}, with index $\frac{5+\sqrt{17}}{2}$ and principal bigraph pair $$\left(\bigraph{bwd1v1v1v1v1v1p1v1x0p0x1v1x0p0x1p0x1v1x0x0v1duals1v1v1v1x2v2x1x3v1}, \bigraph{bwd1v1v1v1v1v1p1v0x1p0x1v0x1v1duals1v1v1v1x2v1}\right)$$ and its dual.
\item The 3311 Goodman-de la Harpe-Jones planar algebra \cite{MR999799}, with index $3+\sqrt{3}$ and principal bigraph pair $$\left(\bigraph{bwd1v1v1v1p1p1v1x0x0v1duals1v1v1x2x3v1}, \bigraph{bwd1v1v1v1p1p1v1x0x0v1duals1v1v1x2x3v1}\right)$$ and its dual (since it is not self-dual despite having the same principal and dual principal graphs \cite{MR1355948}).
\item Izumi's self-dual 2221 planar algebra \cite{MR1832764} and its complex conjugate, with index $\frac{5+\sqrt{21}}{2}$ and principal bigraph pair $$\left(\bigraph{bwd1v1v1p1p1v1x0x0p0x1x0duals1v1v2x1}, \bigraph{bwd1v1v1p1p1v1x0x0p0x1x0duals1v1v2x1}\right).$$
\end{itemize}
\end{thm}

As an immediate corollary (using Popa's results \cite{MR1055708, MR1278111}) this theorem gives a classification of all amenable subfactors of the hyperfinite $II_1$ factor of index less than $5$.  However, the classification of non-amenable subfactors coming from Temperley-Lieb is still wide open (see \cite{MR1159284, MR1293872}).

The choice of $5$ as an upper limit is convenient, but somewhat arbitrary.  In particular, classifying the principal graphs of subfactors at index $5$ does not require a significant amount of new work.  However, at index $5$ there are a large number of possible graphs.  Furthermore, although most of these graphs can be realized via group/subgroup subfactors, they may also be realized in other ways.  This will be addressed in work in joint progress with Izumi.

In this paper, we begin by proving a weaker classification result inspired by Haagerup's original argument.  In particular, we only use the combinatorics of the principal graphs of the subfactors.   We describe several known obstructions to being principal graphs, and a method of enumerating principal graphs satisfying certain size bounds.  Combining these, we obtain various classification results.  Much of the subtlety in this paper lies in finding the right balance between looking for obstructions and extending the enumeration, in order to produce classification results that are both true and useful!  

In order to describe our classification (as well as Haagerup's) it's helpful to define two terms.  We use the term \emph{translation of a graph pair} to indicate a graph pair obtained by increasing the supertransitivity (that is, adding a longer chain of edges at the left) by an even integer, and the term \emph{extension of a graph pair} to indicate a graph pair obtained by extending the graphs in any way at greater depths (i.e. adding vertices and edges at the right).  For example, the $D_{2n}$ principal graph pairs are all translations of $D_{4}$, while the Haagerup principal graph pair is an extension of $D_6$, an extension of $A_4$, a translated extension of $D_4$, and a translated extension of $A_2$.   See \S \ref{sec:notation} for more details.

Haagerup's initial classification is given in the following theorem.
\begin{thm} \cite{MR1317352}
Any subfactor of index between $4$ and $3+\sqrt{3}$ is non-amenable with standard invariant the Temperley-Lieb algebra, or its principal graph pair is a translate of one of the following three graph pairs.
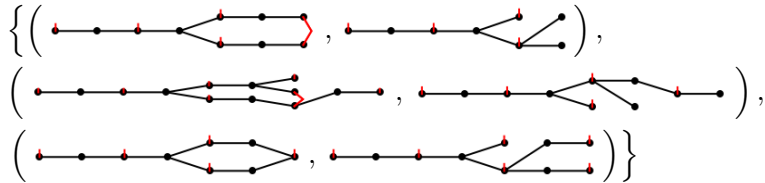
\begin{align*}
&\left \{ \left(\bigraph{bwd1v1v1v1p1v1x0p0x1v1x0p0x1duals1v1v1x2v2x1}, \bigraph{bwd1v1v1v1p1v1x0p1x0duals1v1v1x2}\right), \right . \\
	  &\left(\bigraph{bwd1v1v1v1p1v1x0p0x1v1x0p0x1p0x1v1x0x0v1duals1v1v1x2v2x1x3v1}, \bigraph{bwd1v1v1v1p1v0x1p0x1v0x1v1duals1v1v1x2v1}\right), \\
	  & \left. \left(\bigraph{bwd1v1v1v1p1v1x0p0x1v1x1duals1v1v1x2v1}, \bigraph{bwd1v1v1v1p1v1x0p1x0v1x0p0x1duals1v1v1x2v1x2}\right) \right\}
\end{align*}
\
\end{thm}

Haagerup eliminated all but one graph in the second family using an unpublished connection-based approach, Bisch eliminated the third family completely using fusion algebra techniques \cite{MR1625762}, and Asaeda--Yasuda \cite{MR2472028} eliminated all but two graphs from the first family using arithmetic techniques.

Number theory techniques of Calegari-Morrison-Snyder \cite{1004.0665} (generalizing earlier work of Asaeda-Yasuda \cite{MR2472028}) give an effective bound on how large a translate of a fixed graph can be a principal graph.  Thus any classification along the lines of Haagerup's can now be reduced to a finite list by applying this result.  Unfortunately, in our case the techniques used by Haagerup do not suffice to restrict to the translations of a finite list of graph pairs.   Thus, in addition to a long list of families which can be eliminated using number theoretic techniques, we also have a short list of bad cases.  

\begin{thm}
The principal graphs of any subfactor of index between $4$ and $5$ is a translate of one of an explicit finite list of graph pairs (see Theorem \ref{thm:main}), or is a translated extension of one of the following graphs.
\begin{align*}
\cC &=  \left(\bigraph{bwd1v1v1v1p1v1x0p1x0v1x0v1p1duals1v1v1x2v1}, \bigraph{bwd1v1v1v1p1v1x0p0x1v1x0p1x0p0x1v0x1x0p0x0x1duals1v1v1x2v3x2x1}\right), \displaybreak[1]\\
     \cF &= \left(\bigraph{bwd1v1v1v1p1v1x0p1x0v1x0p0x1v1x0p1x0p0x1v0x1x0p0x0x1v1x0p0x1p0x1duals1v1v1x2v1x2v2x1}, \bigraph{bwd1v1v1v1p1v1x0p0x1v1x0p1x0p0x1p0x1v0x1x0x0p0x0x0x1p1x0x0x0v1x0x0p0x1x0p0x1x0p0x0x1v0x0x0x1p1x0x0x0p0x1x0x0duals1v1v1x2v4x2x3x1v3x2x1x4}\right), \displaybreak[1]\\
      \cB &= \left(\bigraph{bwd1v1v1v1p1v1x0p1x0v1x0p0x1v1x0p0x1v1x0p0x1v1x0p0x1duals1v1v1x2v1x2v2x1}, \bigraph{bwd1v1v1v1p1v1x0p0x1v1x0p1x0p0x1p0x1v0x0x0x1p1x0x0x0v1x0p0x1v0x1p1x0duals1v1v1x2v4x2x3x1v1x2}\right), \displaybreak[1]\\
          \cQ &=\left(\bigraph{bwd1v1v1v1p1p1v1x0x0duals1v1v1x3x2}, \bigraph{bwd1v1v1v1p1p1v1x0x0duals1v1v1x3x2}\right),  \displaybreak[1]\\
              \cQ' &= \left(\bigraph{bwd1v1v1v1p1p1v1x0x0duals1v1v1x2x3}, \bigraph{bwd1v1v1v1p1p1v1x0x0duals1v1v1x2x3}\right)
\end{align*}
\end{thm}

For these remaining cases we need new techniques.  Cases $\cC$, $\cF$, and $\cB$ (which have initial triple points) will be eliminated in joint work with Penneys and Peters \cite{index5-part2}, while $\cQ$ and $\cQ'$ (which have initial quadruple points) will be eliminated in work in progress joint with Izumi and Jones \cite{index5-part3}.  The uniqueness of $2221$ up to conjugation was proved by Han \cite{han-2221}, while the uniqueness of $3331$ up to duality will be proved in \cite{index5-part3}.  Finally, in \cite{index5-part4} Penneys and Tener apply the number theoretic test of  \cite{1004.0665} to eliminate the remaining cases.


Classifications of subfactors of small index may be of interest for several reasons.  We would like to understand whether the appearance of ``exotic" subfactors (like the Haagerup, extended Haagerup, and Asaeda-Haagerup subfactors) is a common phenomenon, or whether these small examples are truly exceptional.  We would like to understand where the boundary takes place between ``small index" and ``large index."  For example, the smallest possible index other than $4$ for an infinite depth subfactor whose standard invariant is not Temperley-Lieb is unknown.  Similarly, we'd like to know the smallest index for which the standard invariant fails to classify subfactors of the hyperfinite factor \cite{MR2314611}.   In order to answer these questions, eventually we would like to extend our classification up to index $3+\sqrt{5}$ where a Fuss-Catalan infinite depth subfactor \cite{MR1437496} appears which may allow for some of the ``large index" behavior found at index $6$.

We would like to thank Richard Burstein, Vaughan Jones, Dave Penneys, and Emily Peters for helpful conversations and careful reading of the manuscript. We also thank all the attendees of the Bodega Bay ``Planar algebra programming camps'' for their interest in this problem: several attendees are writing sequels to this paper.  Kim Morrison was at Microsoft Station Q at UC Santa Barbara and at the Miller Institute for Basic Research at UC Berkeley during this work, and Noah Snyder was supported in part by RTG grant DMS-0354321 and in part by an NSF Postdoctoral Fellowship at Columbia University. We would also like to acknowledge support from the 
DARPA HR0011-11-1-0001 grant.

\section{Enumerating principal graphs}

\subsection{Notation and background}
\label{sec:notation}
Throughout, we use the following definitions.
\begin{defn}
A \emph{bigraph} is a bipartite graph with a specified root vertex.  We allow graphs with infinitely many vertices, but require that every vertex have finite degree.  The \emph{depth} of a vertex is the geodesic distance from the root.  A bigraph is called finite depth if it has finitely many vertices, in which case its depth is the maximum of the depths of the vertices.
\end{defn}

When we draw bigraphs, the root vertex always appears on the left, and the depth of a vertex is always given by its horizontal distance from the root.

\begin{defn}
The \emph{supertransitivity} of a bigraph $\Gamma$ is the greatest integer $n$ so that up to vertices at depth $n$ the graph $\Gamma$ is just $A_n$. Equivalently, it is the number of edges from the root vertex before the first branch point or multiple edge.
\end{defn}

\begin{defn}
A \emph{bigraph with dual data} is a bigraph together with an involution, called duality, of the vertices at each even depth.
\end{defn}

In diagrams, duality is represented by red arcs joining dual pairs of vertices. Self-dual vertices have a small red dash above them.

\begin{defn}
A \emph{bigraph pair} is a pair of bigraphs with dual data, with depths differing by at most one and the same supertransitivity, together with a bijection, also called duality, between the vertices at each odd depth of one graph with the corresponding vertices at the same odd depth on the other graph.
\end{defn}

In diagrams, we do not explicitly indicate the bijections between odd vertices, but rather use the convention that the vertical ordering of vertices at a given depth determines the bijection: the lowest vertices in each graph at each odd depth are dual to each other, etc.

Two bigraph pairs are isomorphic if there is a graph isomorphism between the underlying graphs which preserves the duality structure and the base vertex.  For example, the following two bigraph pairs are isomorphic:
$$ \left(\bigraph{bwd1v1v1v1p1v1x0p0x1v1x0p0x1duals1v1v1x2v2x1}, \bigraph{bwd1v1v1v1p1v1x0p1x0duals1v1v1x2}\right)
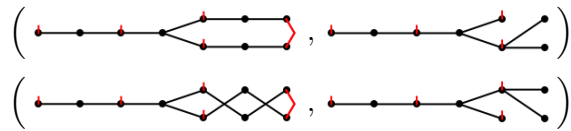
$$
$$ \left(\bigraph{bwd1v1v1v1p1v0x1p1x0v0x1p1x0duals1v1v1x2v2x1}, \bigraph{bwd1v1v1v1p1v0x1p0x1duals1v1v1x2}\right)$$
We say a bigraph pair is \emph{equal} if both graphs have the same depth. Note that if a bigraph pair is unequal then the two depths differ by one, and the longer bigraph needs to have even depth, because of the duality bijection between the odd vertices.

The principal graphs of a subfactor naturally carry the structure of a bigraph pair.  For example, the principal graphs of the Haagerup subfactor are the following pair of bigraphs with dual data.
$$ \left \{ \left(\bigraph{bwd1v1v1v1p1v1x0p0x1v1x0p0x1duals1v1v1x2v2x1}, \bigraph{bwd1v1v1v1p1v1x0p1x0duals1v1v1x2}\right) \right \}$$
We say that a bigraph pair $\Gamma$ is a \emph{translate} of $\Gamma_0$ if there is an even integer $k$ such that up to depth $k$, the bigraphs in $\Gamma$ look like the Dynkin diagram $A_k$, and such that $\Gamma_0$ is the induced bigraphs with dual data produced by removing the first $k-1$ vertices from $\Gamma$ and declaring the unique vertex at depth $k$ the new base vertex.  (The integer $k$ is required to be even because duality is a different kind of structure for odd depths and even depths.)
We say a bigraph pair $\Gamma$ is an \emph{extension} of a equal bigraph pair $\Gamma_0$ if $\Gamma$ is the same as $\Gamma_0$ up to the depth of $\Gamma_0$. We will use the phrase ``$\Gamma$ starts like $\Gamma_0$'' as an abbreviation for ``$\Gamma$ is an extension of a translate of $\Gamma_0$''.

For example, the $D_{2n}$ principal graph pairs are all translations of the $D_{4}$ pair. The Haagerup principal graphs are an extension of $D_6$, an extension of $A_4$, a translated extension of $D_4$, and a translated extension of $A_2$.

Since we are often looking not just at a particular bigraph pair, but instead at all ways of extending and translating it,  we do not insist that the Frobenius-Perron eigenvalues of the two graphs in a bigraph pair are equal  (although they are certainly equal for principal bigraph pairs of finite depth subfactors).  Similarly, although most graphs we discuss are finite depth, since we are considering all translated extensions of these graphs our results apply to infinite depth subfactors.

We also define a sequence of numbers associated to any bigraph.
\begin{defn}
The \emph{annular multiplicities} $\{a_n\}_{n \in \Natural}$ of a bigraph $\Gamma$ are defined by the formula
$$a_n = \sum_{r=0}^n (-1)^{r-n} \frac{2n}{n+r} \binom{n+r}{n-r} w_r$$
where $w_r$ is the number of length $2r$ loops on $\Gamma$ based at the initial vertex.
\end{defn}

Note that the annular multiplicity $a_n$ only depends on the bigraph $\Gamma$ out to depth $n$. Since the annular multiplicities for the $A_l$ graphs are $\{1,0,0,0,\ldots\}$, the annular multiplicities $a_n$ for any $k$-supertransitive graph are $a_1=0$ and $a_n = 1$ for $1 \leq n \leq k$. Thus we will often describe the annular multiplicities by dropping this initial string, and listing the sequence of annular multiplicities starting from the first non-trivial entry.
If a bigraph pair is the principal bigraph pair of a subfactor, then the corresponding planar algebra is a representation of the annular Temperley-Lieb category, and these numbers are in fact the multiplicities of the irreducible representations (c.f. \cite{MR1929335, MR1659204, MR2274519}).
More specifically, the irreducible unitary representations of the annular Temperley-Lieb category at a parameter $\delta$ are parametrized by the set
$$\setc{(0,\mu)}{\mu \in [0,\delta]} \cup \setc{(n,\omega)}{n \in \Natural, n \geq 1, \omega \in \Complex, \omega^n = 1},$$
and the annular multiplcity $a_n$ is the sum of the multiplicities of all irreducible representations of the form $(n,x)$ for some $x$.
In this paper, we only use the annular multiplicities to divide up large collections of graphs, but other papers in this series will make further use of the representation theory of the annular Temperley-Lieb category.

\subsection{Classification statements}
Most of the results of this paper are of the following form:

\begin{quote}
Every subfactor whose principal bigraph pair is not $A_\infty$, which starts like a fixed bigraph pair $\Gamma_0$, and which has index strictly between $4$ and $\Lambda \in \Real$ has principal bigraph pair which either
\begin{enumerate}
\item is a translate (but not an extension!) of one of a certain set of bigraph pairs $\cV$, called ``vines'', or
\item is a translate of an extension of one of a certain set of bigraph pairs $\cW$, called ``weeds''.
\end{enumerate}
\end{quote}

We will abbreviate such a result by saying that the data $(\Gamma_0, \Lambda, \cV, \cW)$ is a ``classification statement''. The idea behind the names is that the `vines' can only grow taller (i.e. increase supertransitivity) but the `weeds' can still grow out of control.

There are of course many boring classification statements, in particular we trivially have $(\Gamma_0, \infty, \eset, \{ \Gamma_0 \})$ for any bigraph pair $\Gamma_0$. The interesting classification statements we produce will have weeds which are much larger than the bigraph pair $\Gamma_0$. In ideal circumstances, we even find classification statements in which the set of weeds is empty.

\begin{example}
In \cite{MR1317352}, Haagerup proves the classification statement
$$((\bigraph{bwd1duals1}, \bigraph{bwd1duals1}), 3+\sqrt{3}, \cV_H, \eset)$$
with
\begin{align*}
\cV_H & = \left \{ \left(\bigraph{bwd1v1v1v1p1v1x0p0x1v1x0p0x1duals1v1v1x2v2x1}, \bigraph{bwd1v1v1v1p1v1x0p1x0duals1v1v1x2}\right) \right . \\
	  & \qquad \left(\bigraph{bwd1v1v1v1p1v1x0p0x1v1x0p0x1p0x1v1x0x0v1duals1v1v1x2v2x1x3v1}, \bigraph{bwd1v1v1v1p1v0x1p0x1v0x1v1duals1v1v1x2v1}\right) \\
	  & \qquad \left. \left(\bigraph{bwd1v1v1v1p1v1x0p0x1v1x1duals1v1v1x2v1}, \bigraph{bwd1v1v1v1p1v1x0p1x0v1x0p0x1duals1v1v1x2v1x2}\right) \right\}
\end{align*}
\end{example}

\begin{rem}
These families of graphs are called the Haagerup family, the Asaeda-Haagerup family, and the hexagon family, respectively.
In Haagerup's paper he claims that that in the Asaeda-Haagerup family only supertransitivity $5$ is possible.  Since Haagerup's paper the hexagon family has been entirely ruled out \cite{MR1625762} and the Haagerup family at supertransitivities $11$ and above have been ruled out \cite{MR2307421,MR2472028}.  A uniform argument for all three cases (and indeed excluding all but finitely many examples coming from any vine) can be given using \cite{1004.0665}.  The $3$-supertransitive bigraph pair in the Haagerup family is realised as the principal bigraph pair of the Haagerup subfactor \cite{MR1686551}, and the $7$-supertransitive bigraph pair is realised by the extended Haagerup subfactor \cite{0909.4099}. The Asaeda-Haagerup bigraph pair is also realised \cite{MR1686551}. In this sense, the classification statement referred to above in terms of vines (and no weeds) has since been turned into a complete classification.  The goal of this paper is to prove a classification statement that can then lead to a complete classification.  Unlike in Haagerup's case, we will still have several weeds in our classification.
\end{rem}

\subsection{The odometer}
\label{sec:odometer}
Given a bigraph $\Gamma$, we can readily enumerate all its extensions with depth one greater and with Frobenius-Perron eigenvalue less than some limit $\Lambda$.  Indeed, the bound on the size of the Frobenius-Perron eigenvalue gives a bound on the valency of each vertex.  In practice, one must be a little careful in order to do this enumeration efficiently.  Suppose we've already enumerated all extensions of $\Gamma$ where the last depth is has rank $k$, and we next want to find all extensions of $\Gamma$ with $k+1$ new vertices.  In terms of the inclusion matrix between the final two depths we are looking for ways of adding a new row to a fixed matrix such that the entries aren't too large. These can be enumerated via an odometer process: increment the first entry of the row until the graph norm is too large, then reset the first entry and increment the second entry, and so on.  In order to further increase efficiency we may assume that the rows of the adjacency matrix between the top two depths are in lexicographic order, as permutations of the rows correspond to permutations of the new vertices, giving isomorphic graphs.


Again with a fixed limit $\Lambda$ on the Frobenius-Perron eigenvalue, we can enumerate all depth one extensions of a bigraph with dual data $\Gamma$ in much the same way. We first forget about dual data and extend the bigraph, then, if the new depth is even, for each extension we consider all possible involutions of the new vertices. We denote the resulting set $\bO_\Lambda(\Gamma)$.

Given a bigraph pair which is equal, we can enumerate all depth one extensions so that both bigraphs have Frobenius-Perron dimension at most $\Lambda$. 
When we are adding an odd depth, we are only interested in pairs where we add the same number of new vertices; in this case, because the duality involution between the new odd vertices is implicitly determined by the order in which they appear, we can insist the rows of the new inclusion matrix are in lexicographic order when extending one of the graphs, but not both. On the other hand, when we are adding an even depth, there may be different numbers of new vertices on the two graphs, and we can ask that both new inclusion matrices have rows in lexicographic order.
This process frequently results in duplicate bigraph pairs which are related by a nontrivial bigraph isomorphism (which may permute vertices at earlier depths as well), and for efficiency these should be removed.

Reusing notation, we denote by $\bO_\Lambda(W)$ the set of all depth one extensions  of a equal bigraph pair $W$. Alternatively, if $W$ has odd depth we can extend one bigraph but not the other, and we call the resulting set $\bL_\Lambda(W)$. When $W$ has an even depth we can't do this (because the numbers of new odd vertices on the new graphs must agree), so $\bL_\Lambda(W) = \eset$. 

The following `meta-theorem' explains how to use these enumeration techniques, which we collectively call `the odometer', to derive new classification statements from old ones.

\begin{thm}[The odometer]
\label{thm:odometer}
Suppose we have a classification statement $$(\Gamma_0, \Lambda, \cV, \{ W \} \cup \cW).$$ Then there is another classification statement $$(\Gamma_0, \Lambda, \cV \cup \{W\} \cup \mathbf{L}_\Lambda(W), \cW \cup \mathbf{O}_ \Lambda(W)).$$
\end{thm}
\begin{proof}
This follows immediately from the definitions.
\end{proof}

Note that by repeatedly applying Theorem \ref{thm:odometer} we can arbitrarily increase the minimum depth of the weeds in a classification statement, at the expense of dramatically increasing the number of vines and weeds.  When we say that we ``run the odometer" what we mean is that we apply Theorem \ref{thm:odometer}, remove all weeds and vines which do not pass the associativity test (explained in the next section), and then repeat.  See Section \ref{sec:running} for details.

Next, we turn to another class of techniques for modifying classification statements: finding obstructions that can rule out entire families of bigraph pairs, coming from either vines or weeds.

\section{Obstructions}

Suppose that $\Gamma$ is a bigraph pair which starts like $\Gamma_0$.  In this section we give several obstructions to $\Gamma$ being the principal bigraph pair of a subfactor.  Furthermore, all of these obstructions are \emph{local} in the sense that they can be computed only from $\Gamma_0$.  Locality in this sense is crucial because it means that these tests can be applied to remove vines and weeds from classification statements.

\subsection{Associativity}
Recall that the principal graph gives the multiplicities for tensoring on the right with ${}_A B_B$ and ${}_B B_A$ between $A-A$ bimodules and $A-B$ bimodules.  The dual principal graph gives the multiplicities for tensoring on the left with ${}_A B_B$ and ${}_B B_A$ between $B-B$ bimodules and $B-A$ bimodules.  However, since taking duals interchanges the order of tensor product, using the dual data we can also recover from a bigraph pair the multiplicities for fusion on the left by the basic bimodules.  (The multiplicities for tensoring with the basic bimodules on both the left and right are often encoded, following Ocneanu, as a $4$-partite graph.  We find bigraph pairs easier to deal with combinatorially and far more compact to display.  However, one can easily go between the two descriptions.)

By associativity, we can first tensor on the left and then on the right, or we can first tensor on the right and then on the left.  This gives a combinatorial obstruction for potential principal graph pairs.  These associativity conditions were first observed by Ocneanu.  In the language of paragroups, associativity becomes the condition that biunitary matrices are square.  In the language of bigraph pairs, associativity becomes the following.

\begin{lem}
\label{lem:associativity}
Suppose that $(\Gamma, \Gamma')$ is the principal bigraph pair of a subfactor and that $\Gamma$ and $\Gamma'$ are simply laced.

Consider two vertices $V$ and $W$ of the same parity in the principal and dual principal graphs respectively.  The following two numbers are equal:
\begin{enumerate}
\item The number of vertices $Z$ in the principal graph which are adjacent to $V$, such that $Z^*$ is adjacent to $W^*$. 
\item The number of vertices $U$ in the dual principal graph which are adjacent to $W$ such that $U^*$ is adjacent to $V^*$. 
\end{enumerate}
\end{lem}

\begin{rem}
If the graphs are not simply laced, the same theorem holds if you count the intermediate vertices with multiplicities corresponding to the product of the multiplicities of the two adjacencies.
\end{rem}



The obstruction given by the above theorem is local in the following sense.  If $V$ and $W$ are not both the largest depth nor both at the smallest depth, then the dimensions computed in either of these fashions are the same for a bigraph pair $\Gamma_0$ as they are for any bigraph pair which starts like $\Gamma_0$.  Furthermore, if $V$ and $W$ are both at the largest depth, then the dimensions computed in either of these fashions are the same for all translates of the given graph.

For example, it is easy to see using the associativity test that if a graph begins like $D_n$ then its dual graph must also begin like $D_n$.

While running the odometer, we can apply a trick which saves some time in checking the associativity obstruction. Recall that when we are adding an even depth, we construct two sets $\bL_\Lambda(W)$ (in which we have extended only one of the two graphs) and $\bO_\Lambda(W)$ (in which both graphs have been extended). It is easy to see that the extended graph in any unequal extension in $\bL_\Lambda(W)$ which passes the (global, in the sense of the paragraph of locality above) associativity test can not also pass the (local) associativity test with any other extended graph in a level extension in $\bO_\Lambda(W)$. 
Suppose $W=(W_1, W_2)$, and $W_1'$ is an extension of $W_1$, $W_2'$ is an extension of $W_2$. Thus if either $(W_1', W_2)$ or $(W_1,W_2')$ satisfies the global associativity test, we can immediately rule out the pair $(W_1',W_2')$ as an element of $\bO_\Lambda(W)$, because it cannot satisfies the local associativity test.



\subsection{The triple point obstruction}
\label{sec:triple-point-obstruction}
Versions of the following results are proved in \cite{MR1317352}, where they are attributed to Ocneanu.


\begin{lem}
Suppose $(\Gamma,\Gamma')$ is the principal bigraph pair of a subfactor of index $4$ or larger, and $V \in \Gamma$ and $V^* \in \Gamma'$ are a pair of dual triple points at an odd depth. Denote by $\cK$ the three neighbours of $V$ on $\Gamma$, and by $\cL$ the three neighbours of $V^*$ on $\Gamma'$. Consider $\phi :  \cK \to \cL$, a dimension preserving bijection. Then there exists $K \in \cK$ and $L \in \cL$ so $\phi(K) \neq L$, and an odd vertex $Z \neq V$ on $\Gamma$ so $Z$ is adjacent to $K$ and $Z^*$ is adjacent to $L$. 
\end{lem}

\begin{lem}
Suppose $(\Gamma,\Gamma')$ is the principal bigraph pair of a subfactor of index $4$ or larger, and $V \in \Gamma$ and $W \in \Gamma'$ are a pair of triple points at an even depth, such that the neighbours of $V$ coincide with the duals of the neighbours of $W^*$, and the duals of the neighbours of $V^*$ coincide with the neighbours of $W$. Denote by $\cK$ the three odd neighbours of $V$ on $\Gamma$, and by $\cL$ the three odd neighbours of $W$ on $\Gamma'$. Consider $\phi :  \cK \to \cL$, a dimension preserving bijection. Then there exists $K \in \cK$ and $L \in \cL$ so $\phi(K) \neq L$, and an even vertex $Z \neq V$ on $\Gamma$ so $Z$ is adjacent to $K$ and $Z^*$ is adjacent to $L^*$. 
\end{lem}

Both of these results are proved in the same way.  Assume for the sake of contradiction that the last condition (the existence of vertex adjacent to $K$ and $L$) does not hold.  By considering the paragroup (or equivalently the 6j-symbols) there is a $3$-by-$3$ unitary matrix each of whose matrix entries have specific absolute values.  A short argument in linear algebra then shows that the index must be less than $4$. See the proof of Proposition 3.5 from \cite{MR1317352} for more details.

In order for these lemmas to give us a local condition, we need to be able to determine which bijections are dimension preserving.  However, under most circumstances it is not possible to compute all the dimensions just from local information.  However, if the triple point in question is at depth $n$ for an $n$-supertransitive subfactor, and if the principal graphs are sufficiently simple through depth $n+2$, then it is possible to apply the triple point obstruction without knowing anything about the graph at higher depths.

First let us fix some terminology.  Suppose that $\cS$ is a finite set, and that $A_1$, $A_2$, $A'_1$, and $A'_2$ are subsets of $\cS$ with $\cS=A_1 \bigcup A_2 = A'_1 \bigcup A'_2$.  We call $(\cS, A_1, A_2, A'_1, A'_2)$ \emph{forbidden} if $A_1$ and $A_2$ are disjoint, $A'_1$ and $A'_2$ are disjoint and either $A_1 = A'_1$ and $A_2 = A'_2$, or $A_1 = A'_2$ and $A_2 = A'_1$.

\begin{cor}
\label{lem:even-triple-point-local}
Suppose that a bigraph pair $(\Gamma, \Gamma')$ has dual initial triple points at an even depth $n$.  Let $\alpha$ and $\beta$ be the vertices of $\Gamma$ at depth $n+1$.  Let $\cS$ be the set of vertices in the principal graph at depth $n+2$.  Let $A_1$ be the set of vertices at depth $n+2$ which are adjacent to $\alpha$ and let $A_2$ be the set of vertices at depth $n+2$ which are adjacent to $\beta$.  Let $A'_i$ be the set of duals of vertices in $A_i$.  If $(\cS, A_1,A_2,A'_1,A'_2)$ is forbidden, then $(\Gamma, \Gamma')$ is not the principal bigraph pair of a subfactor of index $4$ or more.
\end{cor}

\begin{cor}
\label{lem:odd-triple-point-local}
Suppose that a bigraph pair $(\Gamma, \Gamma')$ has dual initial triple points at an odd depth $n$.  Let $\alpha$ and $\beta$ be the two vertices of $\Gamma$ at depth $n+1$ and $\alpha'$, $\beta'$ the two vertices of $\Gamma'$ at depth $n+1$.   Let $\cS$ be the set of vertices of $\Gamma$ at depth $n+2$.  Let $A_1$ be the set of vertices at depth $n+2$ which are adjacent to $\alpha$ and let $A_2$ be the set of vertices at depth $n+2$ which are adjacent to $\beta$.  Let $A'_1$ and $A'_2$ be the duals of the vertices at depth $n+2$ which are adjacent to $\alpha'$ and $\beta'$ respectively.  If $(\cS, A_1,A_2,A'_1,A'_2)$ is forbidden, then $(\Gamma,\Gamma')$ is not the principal bigraph pair of a subfactor of index $4$ or more.
\end{cor}
\begin{proof}
Both of these results follow quickly from the above by using the fact that if two vertices have the same set of adjacent vertices, then they must have the same dimension.
\end{proof}

\subsection{Duals at depth $n+1$}
\label{sec:duals}
Suppose that $\Gamma$ is an $n$-supertransitive bigraph with duals for $n$ odd.  Suppose that $a$ of the vertices at depth $n+1$ are self-dual and that $2b$ of them are not self-dual.  Considering the $180$-degree rotation acting on the quotient of $n+1$-box space by the ideal of elements which factor through the $n-1$-box space, we see that the dimension of the $-1$-eigenspace is $b$ and the dimension of the $+1$-eigenspace is $a+b$.

\begin{lem}
\label{lem:duals}
Suppose that $(\Gamma, \Gamma')$ is the principal bigraph pair of a subfactor, and that $a, b, a', b'$ are defined as above.  Then $a=a'$ and $b=b'$.
\end{lem}
\begin{proof}
Because the half-click rotation gives isomorphisms between the corresponding eigenspaces of the shaded and unshaded $n+1$ box space, we must have $b=b'$ and $a+b=a'+b'$, and hence $(a,b)=(a',b')$.
\end{proof}

\subsection{Even quadruple points}

The following obstruction to quadruple points is proved by Jones using quadratic tangles techniques.  A more sophisticated version of this argument will rule out $\cQ$ entirely.

\begin{thm} \label{thm:evenquad}
The principal graph pair of a subfactor does not begin like $$\left(\bigraph{bwd1v1v1p1p1v1x0x0duals1v1v1}, \bigraph{bwd1v1v1p1p1v1x0x0duals1v1v1}\right).$$
\end{thm}
\begin{proof}
\cite[Theorem 5.2.2.]{quadratic}
\end{proof}

\section{Running the odometer}
\label{sec:running}
By running the odometer on a classification statement we mean repeating the following two procedures:
\begin{enumerate}
\item Apply Theorem \ref{thm:odometer} to the classification statement, extending all the weeds by one depth.
\item Apply the associativity test and remove all vines and all weeds which fail it.
\end{enumerate}

We have written a package of computer programs which automates the odometer, and we describe its use at the end of this section.  It is easy (but very tedious) to check the output of this program by hand.  We've included figures which summarize the output of the odometer as it runs.  These figures are trees, with the initial graph on the left, and each successive equal shows the new weeds that arise that pass the associativity tests.  See below for more details.

The example that we consider reproduces a small part of Haagerup's classification of subfactors out to index $3+\sqrt{3}$ by finding a complete set of vines for principal graph pairs starting like
$$(\Gamma, \Gamma') = \left(\bigraph{bwd1v1v1v1p1v1x0p0x1v1x0p0x1p0x1duals1v1v1x2v2x1x3}, \bigraph{bwd1v1v1v1p1v0x1p0x1v0x1duals1v1v1x2v1}\right)$$
with index less that $3+\sqrt{3}$.

We now ``run the odometer" beginning with the tautological classification statement $\left((\Gamma,\Gamma'), 3+\sqrt{3},\emptyset,\{(\Gamma,\Gamma')\}\right)$.

We first enumerate all the depth $1$ extensions of $\Gamma$ and $\Gamma'$ with norms less than $3+\sqrt{3}$.  Using the notation of \S \ref{sec:odometer}, we have
\begin{align*}
\mathbf{O}_{3+\sqrt{3}}(\Gamma) =\Big\{ 
& \bigraph{bwd1v1v1v1p1v1x0p0x1v1x0p0x1p0x1v0x0x1duals1v1v1x2v2x1x3}, \bigraph{bwd1v1v1v1p1v1x0p0x1v1x0p0x1p0x1v0x1x0duals1v1v1x2v2x1x3}, \\ 
& \bigraph{bwd1v1v1v1p1v1x0p0x1v1x0p0x1p0x1v1x0x0duals1v1v1x2v2x1x3}, 
\bigraph{bwd1v1v1v1p1v1x0p0x1v1x0p0x1p0x1v0x0x1p0x0x1duals1v1v1x2v2x1x3}, \\
& \bigraph{bwd1v1v1v1p1v1x0p0x1v1x0p0x1p0x1v0x1x0p0x0x1duals1v1v1x2v2x1x3}, \bigraph{bwd1v1v1v1p1v1x0p0x1v1x0p0x1p0x1v0x1x0p0x1x0duals1v1v1x2v2x1x3}, \\
&\bigraph{bwd1v1v1v1p1v1x0p0x1v1x0p0x1p0x1v1x0x0p0x0x1duals1v1v1x2v2x1x3}, \bigraph{bwd1v1v1v1p1v1x0p0x1v1x0p0x1p0x1v1x0x0p0x1x0duals1v1v1x2v2x1x3}, \\
&\bigraph{bwd1v1v1v1p1v1x0p0x1v1x0p0x1p0x1v1x0x0p1x0x0duals1v1v1x2v2x1x3}, \bigraph{bwd1v1v1v1p1v1x0p0x1v1x0p0x1p0x1v1x0x0p0x1x0p0x0x1duals1v1v1x2v2x1x3}, \\
& \bigraph{bwd1v1v1v1p1v1x0p0x1v1x0p0x1p0x1v1x0x0p1x0x0p0x0x1duals1v1v1x2v2x1x3}, \bigraph{bwd1v1v1v1p1v1x0p0x1v1x0p0x1p0x1v1x0x0p1x0x0p0x1x0duals1v1v1x2v2x1x3},\\
& \bigraph{bwd1v1v1v1p1v1x0p0x1v1x0p0x1p0x1v1x0x0p1x0x0p0x1x0p0x0x1duals1v1v1x2v2x1x3} \Big\}\\
\mathbf{O}_{3+\sqrt{3}}(\Gamma') = \big\{ 
&\bigraph{bwd1v1v1v1p1v0x1p0x1v0x1v1duals1v1v1x2v1}, \bigraph{bwd1v1v1v1p1v0x1p0x1v0x1v1p1duals1v1v1x2v1}\Big\}
\end{align*}
Now, for the depth one extensions of the bigraph pair $(\Gamma, \Gamma')$, we take one graph from each list, subject to the condition that the two graphs have the same number of vertices at the new (odd) depth.  Since the depth we're adding is an odd depth, there is no need to choose an involution of the new vertices specifying dual data. On the other hand, it is important that in at least one of the lists 
$\mathbf{O}_{3+\sqrt{3}}(\Gamma)$ and $\mathbf{O}_{3+\sqrt{3}}(\Gamma')$ we include as distinct elements graphs which differ by a permutation of the vertices at the new depth,  since the duality for odd vertices is given by their vertical ordering. In the lists above, it is convenient to remove such redundancies in the first list, but not in the second, as there are no redundancies there anyway.

Next, we apply the associativity test of Lemma \ref{lem:associativity}, remembering to only use the local version where at least one of the vertices $V$ and $W$ are not at the newly introduced depth and at least one is not the root vertex. This cuts down the previous large list to just two bigraph pairs,
\begin{align*}
\cW_1 = \Big\{ & \left(\bigraph{bwd1v1v1v1p1v1x0p0x1v1x0p0x1p0x1v1x0x0duals1v1v1x2v2x1x3}, \bigraph{bwd1v1v1v1p1v0x1p0x1v0x1v1duals1v1v1x2v1}\right), \\
    & \left(\bigraph{bwd1v1v1v1p1v1x0p0x1v1x0p0x1p0x1v1x0x0p0x0x1duals1v1v1x2v2x1x3}, \bigraph{bwd1v1v1v1p1v0x1p0x1v0x1v1p1duals1v1v1x2v1}\right) \Big\}.
\end{align*}

As an example of how we rule out all the others, the bigraph pair $$\left(\mathfig{0.3}{associativity-example-1}, \mathfig{0.3}{associativity-example-2}\right)$$ fails the associativity test with the marked vertices: there is one vertex satisfying the conditions for $Z$ in Lemma \ref{lem:associativity}, but none satisfying the conditions for $U$.

Having found all equal extensions, we now look for unequal extensions of $(\Gamma, \Gamma')$ to add to the list of vines and apply the associativity test for vines (where we allow $V$ and $W$ to both be at the largest depth).
Since the new depth we're adding is odd, there can be no unequal extensions.  (However, notice that unequal extensions are easy to enumerate, they're of the form $(\Gamma, \cG')$ for $\cG' \in \mathbf{O}_{3+\sqrt{3}}(\Gamma')$ or of the form $(\cG, \Gamma')$ for $\cG \in \mathbf{O}_{3+\sqrt{3}}(\Gamma)$).  

Finally, the old weed $(\Gamma, \Gamma')$ is added to the list of vines.   So we can apply the associativity test to this as a vine and see that it fails.

At this point we've successfully run the odometer one step.  We now have the classification statement $\left((\Gamma,\Gamma'), 3+\sqrt{3},\emptyset,\cW_1\right)$.
  The progress of the odometer so far is summarized by the following figure.

$$
 \begin{tikzpicture}
\tikzset{grow=right,level distance=150pt}
\tikzset{every tree node/.style={draw,fill=white,rectangle,rounded corners,inner sep=2pt}}
\Tree
[.\node{$\!\!\begin{array}{c}\bigraph{bwd1v1v1v1p1v1x0p0x1v1x0p0x1p0x1duals1v1v1x2v2x1x3}\\\bigraph{bwd1v1v1v1p1v0x1p0x1v0x1duals1v1v1x2v1}\end{array}\!\!$};
	[.\node[fill=red!30]{$\!\!\begin{array}{c}\bigraph{bwd1v1v1v1p1v1x0p0x1v1x0p0x1p0x1v1x0x0duals1v1v1x2v2x1x3}\\\bigraph{bwd1v1v1v1p1v0x1p0x1v0x1v1duals1v1v1x2v1}\end{array}\!\!$};]
	[.\node[fill=red!30]{$\!\!\begin{array}{c}\bigraph{bwd1v1v1v1p1v1x0p0x1v1x0p0x1p0x1v1x0x0p0x0x1duals1v1v1x2v2x1x3}\\\bigraph{bwd1v1v1v1p1v0x1p0x1v0x1v1p1duals1v1v1x2v1}\end{array}\!\!$};]]
\end{tikzpicture}
$$

The current weeds are highlighted in the figure in red.  The old weed $(\Gamma, \Gamma')$ is left in the picture to explain what happened in the odometer at earlier steps (in this case the zeroth step) and is necessary for recovering the vines (see below).

We now run the odometer another step.  This means we go through the above process for each of the two weeds.  One of the weeds has no extensions which pass the associativity test.  The other weed has two extensions.  Furthermore there are two vines which pass, one of which is unequal.

At this point we have proved the classification statement 
\begin{align*}
\Big( (\Gamma,\Gamma'),  3 & +\sqrt{3}, \\
\Big\{ & \left(\bigraph{bwd1v1v1v1p1v1x0p0x1v1x0p0x1p0x1v1x0x0v1duals1v1v1x2v2x1x3v1}, \bigraph{bwd1v1v1v1p1v0x1p0x1v0x1v1duals1v1v1x2v1} \right)  ,\\
&  \left(\bigraph{bwd1v1v1v1p1v1x0p0x1v1x0p0x1p0x1v1x0x0p0x0x1v1x0p1x0p0x1duals1v1v1x2v2x1x3v3x2x1}, \bigraph{bwd1v1v1v1p1v0x1p0x1v0x1v1p1duals1v1v1x2v1} \right) \Big\}\\
 \Big\{&\left(\bigraph{bwd1v1v1v1p1v1x0p0x1v1x0p0x1p0x1v1x0x0p0x0x1v1x1duals1v1v1x2v2x1x3v1}, \bigraph{bwd1v1v1v1p1v0x1p0x1v0x1v1p1v0x1duals1v1v1x2v1v1}\right)  \\
& \left(\bigraph{bwd1v1v1v1p1v1x0p0x1v1x0p0x1p0x1v1x0x0p0x0x1v1x0p1x0p0x1p0x1duals1v1v1x2v2x1x3v4x2x3x1}, \bigraph{bwd1v1v1v1p1v0x1p0x1v0x1v1p1v0x1duals1v1v1x2v1v1} \right) \Big\} \Big)\\
\end{align*}

and the progress of the odometer is summarized by the following figure.
 $$
\scalebox{0.8}{
 \begin{tikzpicture}
\tikzset{grow=right,level distance=160pt}
\tikzset{every tree node/.style={draw,fill=white,rectangle,rounded corners,inner sep=2pt}}
\Tree
[.\node{$\!\!\begin{array}{c}\bigraph{bwd1v1v1v1p1v1x0p0x1v1x0p0x1p0x1duals1v1v1x2v2x1x3}\\\bigraph{bwd1v1v1v1p1v0x1p0x1v0x1duals1v1v1x2v1}\end{array}\!\!$};
	[.\node{$\!\!\begin{array}{c}\bigraph{bwd1v1v1v1p1v1x0p0x1v1x0p0x1p0x1v1x0x0duals1v1v1x2v2x1x3}\\\bigraph{bwd1v1v1v1p1v0x1p0x1v0x1v1duals1v1v1x2v1}\end{array}\!\!$};]
	[.\node{$\!\!\begin{array}{c}\bigraph{bwd1v1v1v1p1v1x0p0x1v1x0p0x1p0x1v1x0x0p0x0x1duals1v1v1x2v2x1x3}\\\bigraph{bwd1v1v1v1p1v0x1p0x1v0x1v1p1duals1v1v1x2v1}\end{array}\!\!$};
		[.\node[fill=red!30]{$\!\!\begin{array}{c}\bigraph{bwd1v1v1v1p1v1x0p0x1v1x0p0x1p0x1v1x0x0p0x0x1v1x1duals1v1v1x2v2x1x3v1}\\\bigraph{bwd1v1v1v1p1v0x1p0x1v0x1v1p1v0x1duals1v1v1x2v1v1}\end{array}\!\!$};]
		[.\node[fill=red!30]{$\!\!\begin{array}{c}\bigraph{bwd1v1v1v1p1v1x0p0x1v1x0p0x1p0x1v1x0x0p0x0x1v1x0p1x0p0x1p0x1duals1v1v1x2v2x1x3v4x2x3x1}\\\bigraph{bwd1v1v1v1p1v0x1p0x1v0x1v1p1v0x1duals1v1v1x2v1v1}\end{array}\!\!$};]]]
\end{tikzpicture}
}
$$

Running the odometer the next step gets us down to a single weed, summarized by the figure 
 $$
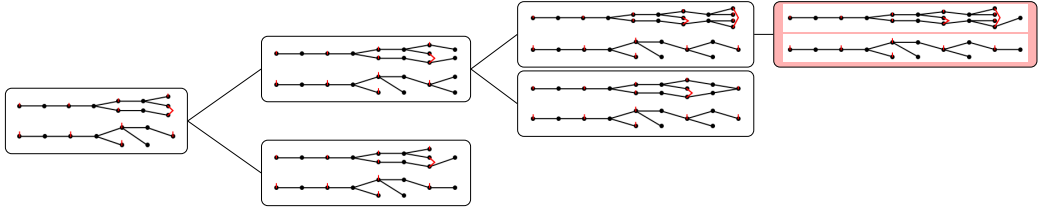
\scalebox{0.6}{
 \begin{tikzpicture}
\tikzset{grow=right,level distance=170pt}
\tikzset{every tree node/.style={draw,fill=white,rectangle,rounded corners,inner sep=2pt}}
\Tree
[.\node{$\!\!\begin{array}{c}\bigraph{bwd1v1v1v1p1v1x0p0x1v1x0p0x1p0x1duals1v1v1x2v2x1x3}\\\bigraph{bwd1v1v1v1p1v0x1p0x1v0x1duals1v1v1x2v1}\end{array}\!\!$};
	[.\node{$\!\!\begin{array}{c}\bigraph{bwd1v1v1v1p1v1x0p0x1v1x0p0x1p0x1v1x0x0duals1v1v1x2v2x1x3}\\\bigraph{bwd1v1v1v1p1v0x1p0x1v0x1v1duals1v1v1x2v1}\end{array}\!\!$};]
	[.\node{$\!\!\begin{array}{c}\bigraph{bwd1v1v1v1p1v1x0p0x1v1x0p0x1p0x1v1x0x0p0x0x1duals1v1v1x2v2x1x3}\\\bigraph{bwd1v1v1v1p1v0x1p0x1v0x1v1p1duals1v1v1x2v1}\end{array}\!\!$};
		[.\node{$\!\!\begin{array}{c}\bigraph{bwd1v1v1v1p1v1x0p0x1v1x0p0x1p0x1v1x0x0p0x0x1v1x1duals1v1v1x2v2x1x3v1}\\\bigraph{bwd1v1v1v1p1v0x1p0x1v0x1v1p1v0x1duals1v1v1x2v1v1}\end{array}\!\!$};]
		[.\node{$\!\!\begin{array}{c}\bigraph{bwd1v1v1v1p1v1x0p0x1v1x0p0x1p0x1v1x0x0p0x0x1v1x0p1x0p0x1p0x1duals1v1v1x2v2x1x3v4x2x3x1}\\\bigraph{bwd1v1v1v1p1v0x1p0x1v0x1v1p1v0x1duals1v1v1x2v1v1}\end{array}\!\!$};
			[.\node[fill=red!30]{$\!\!\begin{array}{c}\bigraph{bwd1v1v1v1p1v1x0p0x1v1x0p0x1p0x1v1x0x0p0x0x1v1x0p1x0p0x1p0x1v1x0x0x0duals1v1v1x2v2x1x3v4x2x3x1}\\\bigraph{bwd1v1v1v1p1v0x1p0x1v0x1v1p1v0x1v1duals1v1v1x2v1v1}\end{array}\!\!$};]]]]
\end{tikzpicture}
}
$$

On the next step of the odometer, there are no weeds that pass the associativity test. At this point the odometer stops, with nothing more to do.  So our final result is the classification statement 
\begin{align*}
\Big( (\Gamma,\Gamma'),  3 & +\sqrt{3}, \\
\Big\{ & \left(\bigraph{bwd1v1v1v1p1v1x0p0x1v1x0p0x1p0x1v1x0x0v1duals1v1v1x2v2x1x3v1}, \bigraph{bwd1v1v1v1p1v0x1p0x1v0x1v1duals1v1v1x2v1} \right) ,\\
& \left(\bigraph{bwd1v1v1v1p1v1x0p0x1v1x0p0x1p0x1v1x0x0p0x0x1v1x0p1x0p0x1duals1v1v1x2v2x1x3v3x2x1}, \bigraph{bwd1v1v1v1p1v0x1p0x1v0x1v1p1duals1v1v1x2v1} \right), \\
& \left(\bigraph{bwd1v1v1v1p1v1x0p0x1v1x0p0x1p0x1v1x0x0p0x0x1v1x1duals1v1v1x2v2x1x3v1}, \bigraph{bwd1v1v1v1p1v0x1p0x1v0x1v1p1v0x1duals1v1v1x2v1v1}\right) \\
& \left( \bigraph{bwd1v1v1v1p1v1x0p0x1v1x0p0x1p0x1v1x0x0p0x0x1v1x0p1x0p0x1p0x1v1x0x0x0v1duals1v1v1x2v2x1x3v4x2x3x1v1}, \bigraph{bwd1v1v1v1p1v0x1p0x1v0x1v1p1v0x1v1duals1v1v1x2v1v1}\right) \Big\}, \\
& \emptyset \Big)
\end{align*}

whose proof is summarized by the following figure.
 $$
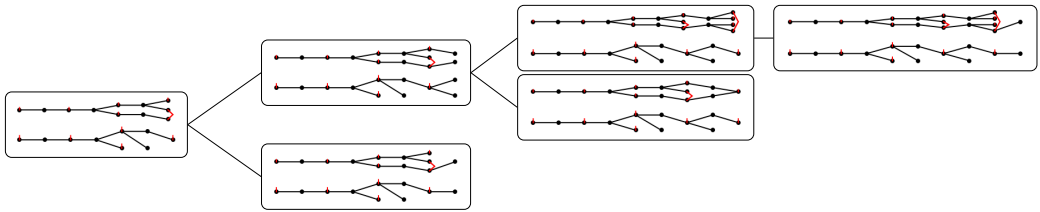
\scalebox{0.6}{
 \begin{tikzpicture}
\tikzset{grow=right,level distance=170pt}
\tikzset{every tree node/.style={draw,fill=white,rectangle,rounded corners,inner sep=2pt}}
\Tree
[.\node{$\!\!\begin{array}{c}\bigraph{bwd1v1v1v1p1v1x0p0x1v1x0p0x1p0x1duals1v1v1x2v2x1x3}\\\bigraph{bwd1v1v1v1p1v0x1p0x1v0x1duals1v1v1x2v1}\end{array}\!\!$};
	[.\node{$\!\!\begin{array}{c}\bigraph{bwd1v1v1v1p1v1x0p0x1v1x0p0x1p0x1v1x0x0duals1v1v1x2v2x1x3}\\\bigraph{bwd1v1v1v1p1v0x1p0x1v0x1v1duals1v1v1x2v1}\end{array}\!\!$};]
	[.\node{$\!\!\begin{array}{c}\bigraph{bwd1v1v1v1p1v1x0p0x1v1x0p0x1p0x1v1x0x0p0x0x1duals1v1v1x2v2x1x3}\\\bigraph{bwd1v1v1v1p1v0x1p0x1v0x1v1p1duals1v1v1x2v1}\end{array}\!\!$};
		[.\node{$\!\!\begin{array}{c}\bigraph{bwd1v1v1v1p1v1x0p0x1v1x0p0x1p0x1v1x0x0p0x0x1v1x1duals1v1v1x2v2x1x3v1}\\\bigraph{bwd1v1v1v1p1v0x1p0x1v0x1v1p1v0x1duals1v1v1x2v1v1}\end{array}\!\!$};]
		[.\node{$\!\!\begin{array}{c}\bigraph{bwd1v1v1v1p1v1x0p0x1v1x0p0x1p0x1v1x0x0p0x0x1v1x0p1x0p0x1p0x1duals1v1v1x2v2x1x3v4x2x3x1}\\\bigraph{bwd1v1v1v1p1v0x1p0x1v0x1v1p1v0x1duals1v1v1x2v1v1}\end{array}\!\!$};
			[.\node{$\!\!\begin{array}{c}\bigraph{bwd1v1v1v1p1v1x0p0x1v1x0p0x1p0x1v1x0x0p0x0x1v1x0p1x0p0x1p0x1v1x0x0x0duals1v1v1x2v2x1x3v4x2x3x1}\\\bigraph{bwd1v1v1v1p1v0x1p0x1v0x1v1p1v0x1v1duals1v1v1x2v1v1}\end{array}\!\!$};]]]]
\end{tikzpicture}
}
$$

To read a classification statement off from a figure do the following:
\begin{itemize}
\item The weeds are the graph pairs highlighted in red.
\item The equal vines are the graph pairs not highlighted which in addition pass the associativity test for pairs of vertices at the largest depth (associativity for other pairs of vertices has already been checked).
\item The unequal vines are found by looking at unequal extensions of the non-highlighted pairs and checking associativity.  Note that if a graph pair passes the full associativity test, then any unequal extension of it automatically \emph{fails} the associativity test.  Furthermore, it is often easy to quickly deduce by hand the unequal extensions which pass the associativity test rather than enumerating then testing all of them.
\end{itemize}
 
Finally, we include a brief tutorial on using the \code{Mathematica} package called \code{FusionAtlas`} (written by the authors along with Dave Penneys, Emily Peters and James Tener) to perform these calculations. First, you'll need a copy of the package. The best way to obtain this is to first install `Subversion', then type at the command line
\begin{quote}
\code{svn checkout http://tqft.net/svn/FusionAtlas/}
\end{quote}
This will create a \code{FusionAtlas} directory in your current directory. Now, in \code{Mathematica}, we need to load the package. First, we add it to the path, with a command like
\begin{mma}
\begin{inm}
AppendTo[\$Path, "~/FusionAtlas/"]
\end{inm}
\end{mma}

(if you downloaded the package somewhere outside your home directory, you'll need to adjust this path). Next, we load the package, with
\begin{mma}
\begin{inm}
<<FusionAtlas`
\end{inm}
\end{mma}

(note the backtick at the end of the line). You should see a message saying the package has been successfully loaded.

The most powerful command for running the odometer is
\begin{quote}
\code{FindBigraphPairExtensionsUpToDepth[L][g1,g2,k]}
\end{quote}
Here \code{L} is the graph norm limit we're working with (the square root of the index), which should be some real number strictly larger than the limit we're really interested in. We've used $\sqrt{5} + 10^{-3}$ throughout. Using a slightly higher limit means that sometimes spurious results will be returned with index above 5, that have to be removed. The parameters \code{g1} and \code{g2} are the bigraphs with dual data that we want to extend. You can find a description of the syntax for specifying bigraphs on the \code{FusionAtlas`} web page, at \url{http://tqft.net/wiki/Atlas_of_subfactors}. All bigraphs have a representation as a string (you can find many examples by looking at the \LaTeX{} source of the article, from the arXiv: all the diagrams are generated automatically from these strings), and you can generate the appropriate expressions in Mathematica from these strings using the function \code{GraphFromString}, and display bigraphs using \code{DisplayBigraph}, for example
\begin{mma}
\begin{inm}
g1 = GraphFromString["bwd1v1v1v1p1v1x0p0x1v1x0p0x1p0x1duals1v1v1x2v2x1x3"]
\end{inm}
\begin{outm}
BigraphWithDuals[
 GradedBigraph[\{\{1\}\}, \{\{1\}\}, \{\{1\}\}, \{\{1\}, \{1\}\}, \{\{1, 0\}, \{0, 1\}\}, \{\{1,
     0\}, \{0, 1\}, \{0, 1\}\}], DualData[\{1\}, \{1\}, \{1, 2\}, \{2, 1, 3\}]]
\end{outm}
\begin{inm}
g2 = GraphFromString["bwd1v1v1v1p1v0x1p0x1v0x1duals1v1v1x2v1"]
\end{inm}
\begin{outm}
BigraphWithDuals[
 GradedBigraph[\{\{1\}\}, \{\{1\}\}, \{\{1\}\}, \{\{1\}, \{1\}\}, \{\{0, 1\}, \{0, 1\}\}, \{\{0,
     1\}\}], DualData[\{1\}, \{1\}, \{1, 2\}, \{1\}]]
\end{outm}
\begin{inm}
DisplayBigraph[g2]
\end{inm}
\begin{outm}
$$\scalebox{2.0}{$\bigraph{bwd1v1v1v1p1v0x1p0x1v0x1duals1v1v1x2v1}$}$$
\end{outm}
\end{mma}

The final parameter \code{k}  specifies the maximum number of times to run the odometer; $\infty$ is an allowed value, although in that case there is no guarantee of termination. The output of \code{FindBigraphPairExtensionsUpToDepth} is a list of three lists. The first list consists of all the vines produced, the second list consists of all the weeds produced, and in this mode the third list is always empty. Thus we can run
\begin{mma}
\begin{inm}
\{vines, weeds, \{\}\} = \\ \hspace{1.8cm}FindBigraphPairExtensionsUpToDepth[Sqrt[3+Sqrt[3]]+$10^{-3}$][g1,g2,$\infty$];
\end{inm}
\begin{inm}
DisplayBigraph /@ vines
\end{inm}
\begin{outm}
\begin{align*}
\Big\{ & \left(\bigraph{bwd1v1v1v1p1v1x0p0x1v1x0p0x1p0x1v1x0x0v1duals1v1v1x2v2x1x3v1}, \bigraph{bwd1v1v1v1p1v0x1p0x1v0x1v1duals1v1v1x2v1} \right) ,\\
& \left(\bigraph{bwd1v1v1v1p1v1x0p0x1v1x0p0x1p0x1v1x0x0p0x0x1v1x0p1x0p0x1duals1v1v1x2v2x1x3v3x2x1}, \bigraph{bwd1v1v1v1p1v0x1p0x1v0x1v1p1duals1v1v1x2v1} \right), \\
& \left(\bigraph{bwd1v1v1v1p1v1x0p0x1v1x0p0x1p0x1v1x0x0p0x0x1v1x1duals1v1v1x2v2x1x3v1}, \bigraph{bwd1v1v1v1p1v0x1p0x1v0x1v1p1v0x1duals1v1v1x2v1v1}\right) \\
& \left( \bigraph{bwd1v1v1v1p1v1x0p0x1v1x0p0x1p0x1v1x0x0p0x0x1v1x0p1x0p0x1p0x1v1x0x0x0v1duals1v1v1x2v2x1x3v4x2x3x1v1}, \bigraph{bwd1v1v1v1p1v0x1p0x1v0x1v1p1v0x1v1duals1v1v1x2v1v1}\right) \Big\}
\end{align*}
\end{outm}
\begin{inm}
weeds
\end{inm}
\begin{outm}
\{\}
\end{outm}
\end{mma}

exactly reproducing the classification described above.

Finally, it is possible to tell  \code{FindBigraphPairExtensionsUpToDepth} to not run the odometer further on some specific weeds, by supplying a fourth argument \code{"Weeds" -> \{ \{h1, h2\}, \{h3, h4\}, ... \}}, where \code{\{h1, h2\}} are a bigraph pair, etc. Now the third list in the output repeats back these weeds which were not run further.

Since this paper was finished, we have written an independent and much faster reimplementation of the odometer in the language \code{Scala}. Happily, it produces all the same results as those described here (and does so in under a minute on a MacBook Pro).  This reimplementation is not yet in a state that others can readily use, but we are happy to share the code and expect that it will soon be more accessible.

\section{$1$-supertransitive subfactors}
In this section we use ad hoc methods to eliminate all $1$-supertransitive subfactors of index between $4$ and $5$.  This reduces the size of the odometer output drastically, simply because there are far fewer 3-supertransitive graphs below index $5$ than there are $1$-supertransitive graphs below index $5$.  The crux of the proof is that no subfactor of index between $4$ and $5$ can have an intermediate subfactor.

\begin{thm}
\label{thm:intermediate}
There are no $1$-supertransitive subfactors with index strictly between $4$ and $5$.
\end{thm}
\begin{proof}
Suppose that all of the objects at depth $2$ have dimension $1$.  Then the index is an integer, and so does not lie between $4$ and $5$.  Suppose that some object at depth $2$ has dimension $1$, but that not all objects at depth $2$ have dimension $1$.  Then there must exist an intermediate subfactor.  However, there are no numbers between $4$ and $5$ which are the product of two allowed indices.

Let $X$ be the fundamental object.  Suppose that there is an object $V$ at depth $2$ with dimension bigger than $1$ but less than $2$.  Consider the connected component of $X$ in the fusion graph for tensoring on the left with $V$. (Be careful here: usually we talk about principal graphs for tensoring on the right with an object.)  The Frobenius-Perron eigenvalue for this graph is $\dim V < 2$, hence  this graph must be an ADET type graph.  Furthermore, the Frobenius-Perron eigenvector is, up to scaling, given by the dimensions of the fundamental object and the other objects which come up in the fusion graph.  From the principal graph we see that there's a non-zero map from $X \otimes X^* \rightarrow V$.  Hence by Frobenius reciprocity, there's a nonzero map $X \rightarrow V \tensor X$.  Therefore in the fusion graph, $X$ must have a self-loop.  Hence $X$ must be at the loop end of a type $T$ graph.  Let $Y_0, Y_1, \ldots $ be the other objects in this graph with $Y_0$ all the way at the non-loop end.  Then the normalized Frobenius-Perron eigenvector is $(1, \dim Y_1 / \dim Y_0, \ldots, \dim X / \dim Y_0)$.  If the fusion graph is $T_2$ we see that $\dim X / \dim Y_0 = (1+\sqrt{5})/2$.  If $\dim Y_0 = 1$, then the index, $(\dim X)^2$, is less than $4$, while if $\dim Y_0$ is at least $\sqrt{2}$ then $\dim X$ is larger than $5$.  If the fusion graph is $T_k$ for $k>2$, then $(\dim X)^2 > (\dim X / \dim Y_0)^2 > 5$.

Thus there are at least two objects at depth $2$ both of which have dimension at least $2$, hence the index is at least $1+2+2 = 5$.
\end{proof}

We note that in order to extend our work here to include index equal to $5$, one would have to classify  by hand the $1$-supertransitive graphs at index exactly $5$, by extending the methods of the proof here.


\section{Main result}
We now use the techniques described in the three previous sections to develop classification statements for all subfactors with index between $4$ and $5$. In particular, we obtain the following classification statement with a `manageable' set of weeds.

\begin{thm}
\label{thm:main}
Subfactors with index between $4$ and $5$ are described by the classification statement $$((\bigraph{bwd1duals1}, \bigraph{bwd1duals1}), 5, \cV_\infty, \cW_\infty)$$ with
\begin{align*}
\cV_\infty  = \Big \{
    & \left(\bigraph{bwd1v1v1p1p1duals1v1}, \bigraph{bwd1v1v1p1p1duals1v1}\right), \displaybreak[1]\\
    & \left(\bigraph{bwd1v1v1v1p1p1duals1v1v1x2x3}, \bigraph{bwd1v1v1v1p1p1duals1v1v1x2x3}\right), \displaybreak[1]\\
    & \left(\bigraph{bwd1v1v1v1p1p1duals1v1v3x2x1}, \bigraph{bwd1v1v1v1p1p1duals1v1v3x2x1}\right), \displaybreak[1]\\
    & \left(\bigraph{bwd1v1v1p1v1x1duals1v1v1}, \bigraph{bwd1v1v1p1v1x1duals1v1v1}\right), \displaybreak[1]\\
    & \left(\bigraph{bwd1v1v1p1v1x0p1x0p0x1p0x1duals1v1v4x2x3x1}, \bigraph{bwd1v1v1p1v1x1duals1v1v1}\right), \displaybreak[1]\\
    & \left(\bigraph{bwd1v1v1p1p1v1x0x0p0x1x0duals1v1v1x2}, \bigraph{bwd1v1v1p1p1v1x0x0p0x1x0duals1v1v1x2}\right), \displaybreak[1]\\
    & \left(\bigraph{bwd1v1v1p1p1v1x0x0p0x1x0duals1v1v2x1}, \bigraph{bwd1v1v1p1p1v1x0x0p0x1x0duals1v1v2x1}\right), \displaybreak[1]\\
    & \left(\bigraph{bwd1v1v1p1p1v1x0x0p1x0x0duals1v1v1x2}, \bigraph{bwd1v1v1p1p1v1x0x0p1x0x0duals1v1v1x2}\right), \displaybreak[1]\\
    & \left(\bigraph{bwd1v1v1p1p1v1x0x0p1x0x0duals1v1v2x1}, \bigraph{bwd1v1v1p1p1v1x0x0p1x0x0duals1v1v1x2}\right), \displaybreak[1]\\
    & \left(\bigraph{bwd1v1v1p1p1v1x0x0p1x0x0duals1v1v2x1}, \bigraph{bwd1v1v1p1p1v1x0x0p1x0x0duals1v1v2x1}\right), \displaybreak[1]\\
    & \left(\bigraph{bwd1v1v1v1p1v1x1duals1v1v1x2}, \bigraph{bwd1v1v1v1p1v1x1duals1v1v1x2}\right), \displaybreak[1]\\
    & \left(\bigraph{bwd1v1v1v1p1v1x1duals1v1v2x1}, \bigraph{bwd1v1v1v1p1v1x1duals1v1v2x1}\right), \displaybreak[1]\\
    & \left(\bigraph{bwd1v1v1v1p1p1v1x0x0p0x0x1duals1v1v3x2x1}, \bigraph{bwd1v1v1v1p1p1v1x0x0p0x0x1duals1v1v3x2x1}\right), \displaybreak[1]\\
    & \left(\bigraph{bwd1v1v1v1p1p1v1x0x0p0x1x0duals1v1v1x2x3}, \bigraph{bwd1v1v1v1p1p1v1x0x0p0x1x0duals1v1v1x2x3}\right), \displaybreak[1]\\
    & \left(\bigraph{bwd1v1v1p1p1v1x0x0p0x1x0v1x0p0x1duals1v1v2x1}, \bigraph{bwd1v1v1p1p1v1x0x0p0x1x0v1x0p0x1duals1v1v2x1}\right), \displaybreak[1]\\
    & \left(\bigraph{bwd1v1v1v1p1v1x0p0x1v1x0p0x1duals1v1v1x2v2x1}, \bigraph{bwd1v1v1v1p1v1x0p1x0duals1v1v1x2}\right), \displaybreak[1]\\
    & \left(\bigraph{bwd1v1v1v1p1v1x0p1x0v1x0p0x1duals1v1v1x2v1x2}, \bigraph{bwd1v1v1v1p1v1x0p0x1v1x1duals1v1v1x2v1}\right), \displaybreak[1]\\
    & \left(\bigraph{bwd1v1v1v1p1v1x0p1x0v1x0p1x0p0x1duals1v1v1x2v1x2x3}, \bigraph{bwd1v1v1v1p1v1x0p0x1v1x1p1x0duals1v1v1x2v1x2}\right), \displaybreak[1]\\
    & \left(\bigraph{bwd1v1v1v1p1v1x0p1x0v1x0p1x0p0x1duals1v1v1x2v2x1x3}, \bigraph{bwd1v1v1v1p1v1x0p0x1v1x1p1x0duals1v1v1x2v1x2}\right), \displaybreak[1]\\
    & \left(\bigraph{bwd1v1v1v1p1v1x0p1x0p0x1v1x0x0p0x1x0p0x0x1p0x0x1duals1v1v1x2v4x3x2x1}, \bigraph{bwd1v1v1v1p1v1x0p1x0p1x0duals1v1v1x2}\right), \displaybreak[1]\\
    & \left(\bigraph{bwd1v1v1p1v1x0p1x0p0x1v0x1x0p0x0x1v0x1duals1v1v3x2x1v1}, \bigraph{bwd1v1v1p1v1x0p1x0p0x1v0x0x1p0x1x0v1x0duals1v1v3x2x1v1}\right), \displaybreak[1]\\
    & \left(\bigraph{bwd1v1v1p1p1v1x0x0p0x1x0v1x0p0x1v1x0p0x1duals1v1v2x1v2x1}, \bigraph{bwd1v1v1p1p1v1x0x0p0x1x0v1x0p0x1v1x0p0x1duals1v1v2x1v2x1}\right), \displaybreak[1]\\
    & \left(\bigraph{bwd1v1v1v1v1p1p1v1x0x0p0x1x0v1x0duals1v1v1v1x2}, \bigraph{bwd1v1v1v1v1p1p1v1x0x0p0x1x0v1x0duals1v1v1v1x2}\right), \displaybreak[1]\\
    & \left(\bigraph{bwd1v1v1v1p1v1x0p1x0v1x0p1x0p0x1v1x0x0duals1v1v1x2v1x2x3}, \bigraph{bwd1v1v1v1p1v1x0p0x1v1x1p1x0v0x1duals1v1v1x2v1x2}\right), \displaybreak[1]\\
    & \left(\bigraph{bwd1v1v1v1p1v1x0p1x0p0x1v1x0x0p0x0x1v0x1duals1v1v1x2v2x1}, \bigraph{bwd1v1v1v1p1v1x0p0x1p1x0v1x0x0p0x1x0v0x1duals1v1v1x2v2x1}\right), \displaybreak[1]\\
    & \left(\bigraph{bwd1v1v1v1v1v1p1p1v1x0x0p0x1x0v1x0duals1v1v1v1x2x3v1}, \bigraph{bwd1v1v1v1v1v1p1p1v1x0x0p0x1x0v1x0duals1v1v1v1x2x3v1}\right), \displaybreak[1]\\
    & \left(\bigraph{bwd1v1v1v1v1p1p1v1x0x0p0x1x0v1x0v1duals1v1v1v1x2v1}, \bigraph{bwd1v1v1v1v1p1p1v1x0x0p0x1x0v1x0v1duals1v1v1v1x2v1}\right), \displaybreak[1]\\
    & \left(\bigraph{bwd1v1v1v1p1v1x0p0x1v1x1v1v1duals1v1v1x2v1v1}, \bigraph{bwd1v1v1v1p1v1x0p1x0v1x0p0x1v1x1duals1v1v1x2v1x2}\right), \displaybreak[1]\\
    & \left(\bigraph{bwd1v1v1v1p1v1x0p0x1v1x0p1x0p0x1v0x0x1v1duals1v1v1x2v3x2x1v1}, \bigraph{bwd1v1v1v1p1v1x0p1x0v1x0v1duals1v1v1x2v1}\right), \displaybreak[1]\\
    & \left(\bigraph{bwd1v1v1v1p1v1x0p0x1v1x0p1x0p0x1v0x1x0p0x1x0p0x0x1v1x0x0p0x1x0p0x0x1p0x0x1p0x0x1duals1v1v1x2v3x2x1v5x4x3x2x1}, \bigraph{bwd1v1v1v1p1v1x0p1x0v1x0v1p1p1duals1v1v1x2v1}\right), \displaybreak[1]\\
    & \left(\bigraph{bwd1v1v1v1p1v1x0p0x1v1x0p1x0p1x0p0x1v0x0x0x1p0x0x0x1v1x0p0x1duals1v1v1x2v4x2x3x1v1x2}, \bigraph{bwd1v1v1v1p1v1x0p1x0v1x0p1x0v1x0p0x1duals1v1v1x2v1x2}\right), \displaybreak[1]\\
    & \left(\bigraph{bwd1v1v1v1p1v1x0p0x1v1x0p1x0p1x0p0x1v0x0x0x1p0x0x0x1v1x0p0x1duals1v1v1x2v4x2x3x1v2x1}, \bigraph{bwd1v1v1v1p1v1x0p1x0v1x0p1x0v1x0p0x1duals1v1v1x2v2x1}\right), \displaybreak[1]\\
    & \left(\bigraph{bwd1v1v1v1p1v1x0p0x1v1x0p1x0p1x0p0x1v0x0x0x1p0x0x0x1v1x0p0x1duals1v1v1x2v4x3x2x1v1x2}, \bigraph{bwd1v1v1v1p1v1x0p1x0v1x0p1x0v1x0p0x1duals1v1v1x2v1x2}\right), \displaybreak[1]\\
    & \left(\bigraph{bwd1v1v1v1p1v1x0p0x1v1x0p1x0p1x0p0x1v0x0x0x1p0x0x0x1v1x0p0x1duals1v1v1x2v4x3x2x1v2x1}, \bigraph{bwd1v1v1v1p1v1x0p1x0v1x0p1x0v1x0p0x1duals1v1v1x2v2x1}\right), \displaybreak[1]\\
    & \left(\bigraph{bwd1v1v1v1p1v1x0p1x0v1x0p0x1v1x0p0x1v1x1duals1v1v1x2v1x2v1}, \bigraph{bwd1v1v1v1p1v1x0p0x1v1x0p1x0p0x1p0x1v0x0x0x1p1x0x0x0v1x0p1x0p0x1p0x1duals1v1v1x2v4x2x3x1v1x2x3x4}\right), \displaybreak[1]\\
    & \left(\bigraph{bwd1v1v1v1p1v1x0p1x0v1x0p0x1v1x0p0x1v1x1duals1v1v1x2v1x2v1}, \bigraph{bwd1v1v1v1p1v1x0p0x1v1x0p1x0p0x1p0x1v0x0x0x1p1x0x0x0v1x0p1x0p0x1p0x1duals1v1v1x2v4x2x3x1v2x1x3x4}\right), \displaybreak[1]\\
    & \left(\bigraph{bwd1v1v1v1p1v1x0p1x0v1x0p0x1v1x0p0x1v1x1duals1v1v1x2v1x2v1}, \bigraph{bwd1v1v1v1p1v1x0p0x1v1x0p1x0p0x1p0x1v0x0x0x1p1x0x0x0v1x0p1x0p0x1p0x1duals1v1v1x2v4x2x3x1v2x1x4x3}\right), \displaybreak[1]\\
    & \left(\bigraph{bwd1v1v1v1p1v1x0p1x0v1x0p0x1v1x0p0x1v1x0p1x0p0x1p0x1duals1v1v1x2v1x2v4x3x2x1}, \bigraph{bwd1v1v1v1p1v1x0p0x1v1x0p1x0p0x1p0x1v0x0x0x1p1x0x0x0v1x1duals1v1v1x2v4x2x3x1v1}\right), \displaybreak[1]\\
    & \left(\bigraph{bwd1v1v1v1v1v1p1p1v1x0x0p0x1x0v1x0v1duals1v1v1v1x2x3v1}, \bigraph{bwd1v1v1v1v1v1p1p1v1x0x0p0x1x0v1x0v1duals1v1v1v1x2x3v1}\right), \displaybreak[1]\\
    & \left(\bigraph{bwd1v1v1v1v1p1p1v1x0x0p0x1x0v1x0v1v1duals1v1v1v1x2v1}, \bigraph{bwd1v1v1v1v1p1p1v1x0x0p0x1x0v1x0v1v1duals1v1v1v1x2v1}\right), \displaybreak[1]\\
    & \left(\bigraph{bwd1v1v1v1p1v1x0p0x1v1x0p1x0p0x1p0x1v0x0x0x1p1x0x0x0v1x1v1v1duals1v1v1x2v4x2x3x1v1v1}, \bigraph{bwd1v1v1v1p1v1x0p1x0v1x0p0x1v1x0p0x1v1x0p1x0p0x1p0x1v1x0x0x1duals1v1v1x2v1x2v4x3x2x1}\right), \displaybreak[1]\\
    & \left(\bigraph{bwd1v1v1v1p1v1x0p1x0v1x0p0x1v1x0p0x1v1x1v1v1duals1v1v1x2v1x2v1v1}, \bigraph{bwd1v1v1v1p1v1x0p0x1v1x0p1x0p0x1p0x1v0x0x0x1p1x0x0x0v1x0p1x0p0x1p0x1v1x0x1x0duals1v1v1x2v4x2x3x1v1x2x3x4}\right), \displaybreak[1]\\
    & \left(\bigraph{bwd1v1v1v1p1v1x0p1x0p0x1v1x0x0p0x1x0p0x0x1v0x1x0p0x0x1v1x0p0x1v0x1v1duals1v1v1x2v3x2x1v2x1v1}, \bigraph{bwd1v1v1v1p1v1x0p0x1p1x0v1x0x0p0x1x0p0x1x0v1x0x0p0x0x1v1x0v1duals1v1v1x2v3x2x1v1}\right) \displaybreak[1]
 \Big \} \\
\intertext{and}
\cW_\infty  = \Big\{    \cC &=  \left(\bigraph{bwd1v1v1v1p1v1x0p1x0v1x0v1p1duals1v1v1x2v1}, \bigraph{bwd1v1v1v1p1v1x0p0x1v1x0p1x0p0x1v0x1x0p0x0x1duals1v1v1x2v3x2x1}\right), \displaybreak[1]\\
     \cF &= \left(\bigraph{bwd1v1v1v1p1v1x0p1x0v1x0p0x1v1x0p1x0p0x1v0x1x0p0x0x1v1x0p0x1p0x1duals1v1v1x2v1x2v2x1}, \bigraph{bwd1v1v1v1p1v1x0p0x1v1x0p1x0p0x1p0x1v0x1x0x0p0x0x0x1p1x0x0x0v1x0x0p0x1x0p0x1x0p0x0x1v0x0x0x1p1x0x0x0p0x1x0x0duals1v1v1x2v4x2x3x1v3x2x1x4}\right), \displaybreak[1]\\
      \cB &= \left(\bigraph{bwd1v1v1v1p1v1x0p1x0v1x0p0x1v1x0p0x1v1x0p0x1v1x0p0x1duals1v1v1x2v1x2v2x1}, \bigraph{bwd1v1v1v1p1v1x0p0x1v1x0p1x0p0x1p0x1v0x0x0x1p1x0x0x0v1x0p0x1v0x1p1x0duals1v1v1x2v4x2x3x1v1x2}\right), \displaybreak[1]\\
      \cQ &=\left(\bigraph{bwd1v1v1v1p1p1v1x0x0duals1v1v1x3x2}, \bigraph{bwd1v1v1v1p1p1v1x0x0duals1v1v1x3x2}\right)\\
    \cQ' &= \left(\bigraph{bwd1v1v1v1p1p1v1x0x0duals1v1v1x2x3}, \bigraph{bwd1v1v1v1p1p1v1x0x0duals1v1v1x2x3}\right)
	\Big\}
\end{align*}
\end{thm}

The vines can all be eliminated (with four exceptions: Haagerup, Asaeda-Haagerup, extended Haagerup, and 2221) by showing that the indices are non-cyclotomic by applying the results of \cite{1004.0665}.  This will be done in an upcoming paper \cite{index5-part4} by Dave Penneys and James Tener.  Furthermore, a forthcoming joint paper\cite{index5-part2} with Dave Penneys and Emily Peters shows that there are no subfactors with principal graphs starting like $\cB$, $\cC$ or $\cF$, at any index. A forthcoming paper \cite{index5-part3} joint with Masaki Izumi and Vaughan Jones shows that there are no subfactors with principal graphs starting like $\cQ$ or $\cQ'$ at any index other than the $3331$ principal graph.


The rest of this section is dedicated to the proof of Theorem \ref{thm:main}. First, in \S \ref{sec:initial-seeds} we produce an initial list of weeds, which begin with either a triple point or a quadruple point. In \S \ref{sec:triple-points}, we then run the odometer for a single step extending all of the triple point weeds by one depth. We then apply the triple point obstruction from \S \ref{sec:triple-point-obstruction} to rule out many of the resulting weeds. We then run the odometer on all the surviving triple points weeds. In most cases we can iterate this until no more weeds survive, but for one weed with annular multiplicities $10$ (the one responsible for the Haagerup, extended Haagerup, and Asaeda-Haagerup subfactors) we have to stop the odometer by hand, leaving a set of three more complicated weeds. In \S \ref{sec:quadruple-points} we run the odometer on all the quadruple point weeds, stopping again with a more complicated set of weeds. In \S \ref{sec:killing-quadruple-weeds} we rule out certain of these weeds by hand. The full list of vines and weeds in Theorem \ref{thm:main} above is assembled out of all the vines and weeds produced in the various subsections.

\subsection{Initial seeds}
\label{sec:initial-seeds}

\begin{lem}
We have the classification statement
$$((\bigraph{bwd1duals1}, \bigraph{bwd1duals1}), 5, \eset, \cW_1)$$ with
\begin{align*}
\cW_1 = \Big\{ \Gamma_{o1,a} & = \left( \bigraph{bwd1v1v1v1p1duals1v1v1x2}, \bigraph{bwd1v1v1v1p1duals1v1v1x2} \right)  \displaybreak[1] \\
	    \Gamma_{o1,b}& = \left( \bigraph{bwd1v1v1v1p1duals1v1v1x2}, \bigraph{bwd1v1v1v1p1duals1v1v2x1} \right) \displaybreak[1] \\
	   \Gamma_{o1,c} &= \left( \bigraph{bwd1v1v1v1p1duals1v1v2x1}, \bigraph{bwd1v1v1v1p1duals1v1v2x1} \right) \displaybreak[1] \\
	    \Gamma_{e1\phantom{,a}} &= \left( \bigraph{bwd1v1v1p1duals1v1}, \bigraph{bwd1v1v1p1duals1v1} \right) \displaybreak[1] \\
	   \Gamma_{o2,a} &= \left( \bigraph{bwd1v1v1v1p1p1duals1v1v1x2x3}, \bigraph{bwd1v1v1v1p1p1duals1v1v1x2x3} \right) \displaybreak[1] \\
	   \Gamma_{o2,b} &= \left( \bigraph{bwd1v1v1v1p1p1duals1v1v1x2x3}, \bigraph{bwd1v1v1v1p1p1duals1v1v2x1x3} \right) \displaybreak[1] \\
	     \Gamma_{o2,c} &= \left( \bigraph{bwd1v1v1v1p1p1duals1v1v2x1x3}, \bigraph{bwd1v1v1v1p1p1duals1v1v2x1x3} \right) \displaybreak[1] \\
	     \Gamma_{e2\phantom{,a}} &=  \left( \bigraph{bwd1v1v1p1p1duals1v1}, \bigraph{bwd1v1v1p1p1duals1v1} \right) \Big\}
\end{align*}
\end{lem}
\begin{proof}
By Theorem \ref{thm:intermediate}, we can assume that any odd-supertransitive bigraphs with index less than $5$ are at least $3$-supertransitive.  Since the $5$-pointed star and the vertex with a single and double edge coming into it both have norm $5$, it follows that there are only two possibilities at the first non-trivial depth: we have either a vertex of valence $3$ or a vertex of valence $4$.  By the associativity test both graphs must have the same valency.  
\end{proof}

The naming scheme $\Gamma_{xn,m}$ indicates whether the branch point is at an odd or even depth, according to whether $x=o$ or $x=e$, the number $n$ denotes the annular multiplicity  at the depth one past the branch point, and then $m$ is an arbitrary alphabetic index.

We immediately rule out the weeds $\Gamma_{o1,b}$ and $\Gamma_{o2,b}$ on the basis of Lemma \ref{lem:duals}.

The proof of Theorem \ref{thm:main} will follow by successively applying the odometer and removing graphs which fail the obstructions.  For ease of exposition we will split up into cases based on the form of the initial branch point.

\subsection{Triple points}
\label{sec:triple-points}

We now run the odometer for one step on each of $\Gamma_{o1,a}$, $\Gamma_{o1,c}$ and $\Gamma_{e1}$, obtaining 

\begin{lem}
There are classification statements
\begin{align}
\label{eq:o1a}
\Big(\Gamma_{o1,a}, 5, & \left\{\Gamma_{o1,a}\right\}, \\
 \Big\{
    & \left(\bigraph{bwd1v1v1v1p1v1x0p1x0p0x1p0x1duals1v1v1x2}, \bigraph{bwd1v1v1v1p1v1x0p0x1p1x0p0x1duals1v1v1x2}\right), \displaybreak[1]\notag\\
    & \left(\bigraph{bwd1v1v1v1p1v1x0p1x0p1x0duals1v1v1x2}, \bigraph{bwd1v1v1v1p1v1x0p1x0p0x1duals1v1v1x2}\right), \displaybreak[1]\notag\\
    & \left(\bigraph{bwd1v1v1v1p1v1x0p1x0p0x1duals1v1v1x2}, \bigraph{bwd1v1v1v1p1v1x0p0x1p1x0duals1v1v1x2}\right), \displaybreak[1]\notag\\
    & \left(\bigraph{bwd1v1v1v1p1v1x0p1x0duals1v1v1x2}, \bigraph{bwd1v1v1v1p1v1x0p0x1duals1v1v1x2}\right), \displaybreak[1]\notag\\
    & \left(\bigraph{bwd1v1v1v1p1v1x1duals1v1v1x2}, \bigraph{bwd1v1v1v1p1v1x1duals1v1v1x2}\right), \displaybreak[1]\notag\\
    & \left(\bigraph{bwd1v1v1v1p1v1x0p1x0p0x1p0x1duals1v1v1x2}, \bigraph{bwd1v1v1v1p1v1x0p1x0p0x1p0x1duals1v1v1x2}\right), \displaybreak[1]\notag\\
    & \left(\bigraph{bwd1v1v1v1p1v1x0p1x0p1x0duals1v1v1x2}, \bigraph{bwd1v1v1v1p1v1x0p1x0p1x0duals1v1v1x2}\right), \displaybreak[1]\notag\\
    & \left(\bigraph{bwd1v1v1v1p1v1x0p1x0p0x1duals1v1v1x2}, \bigraph{bwd1v1v1v1p1v1x0p1x0p0x1duals1v1v1x2}\right), \displaybreak[1]\notag\\
    & \left(\bigraph{bwd1v1v1v1p1v1x0p1x0duals1v1v1x2}, \bigraph{bwd1v1v1v1p1v1x0p1x0duals1v1v1x2}\right), \displaybreak[1]\notag\\
    & \left(\bigraph{bwd1v1v1v1p1v1x0p0x1duals1v1v1x2}, \bigraph{bwd1v1v1v1p1v1x0p0x1duals1v1v1x2}\right), \displaybreak[1]\notag\\
    & \left(\bigraph{bwd1v1v1v1p1v1x0duals1v1v1x2}, \bigraph{bwd1v1v1v1p1v1x0duals1v1v1x2}\right) \displaybreak[1]\notag
\Big\} \Big), \notag\\
\label{eq:o1c}
\Big(\Gamma_{o1,c}, 5, & \left\{\Gamma_{o1,c}\right\}, \\
 \Big\{
    & \left(\bigraph{bwd1v1v1v1p1v1x0p1x0p0x1p0x1duals1v1v2x1}, \bigraph{bwd1v1v1v1p1v1x0p0x1p1x0p0x1duals1v1v2x1}\right), \displaybreak[1]\notag\\
    & \left(\bigraph{bwd1v1v1v1p1v1x1duals1v1v2x1}, \bigraph{bwd1v1v1v1p1v1x1duals1v1v2x1}\right), \displaybreak[1]\notag\\
    & \left(\bigraph{bwd1v1v1v1p1v1x0p1x0p0x1p0x1duals1v1v2x1}, \bigraph{bwd1v1v1v1p1v1x0p1x0p0x1p0x1duals1v1v2x1}\right), \displaybreak[1]\notag\\
    & \left(\bigraph{bwd1v1v1v1p1v1x0p0x1duals1v1v2x1}, \bigraph{bwd1v1v1v1p1v1x0p0x1duals1v1v2x1}\right) \displaybreak[1]\notag
\Big\} \Big) \notag\\
\intertext{and}
\label{eq:e1}
\Big(\Gamma_{e1}, 5, & \left\{\Gamma_{e1}\right\}, \\
 \Big\{
    & \left(\bigraph{bwd1v1v1p1v1x0p1x0p0x1p0x1duals1v1v4x2x3x1}, \bigraph{bwd1v1v1p1v1x0p1x0p0x1p0x1duals1v1v4x2x3x1}\right), \displaybreak[1]\notag\\
    & \left(\bigraph{bwd1v1v1p1v1x0p1x0p0x1p0x1duals1v1v4x2x3x1}, \bigraph{bwd1v1v1p1v1x1duals1v1v1}\right), \displaybreak[1]\notag\\
    & \left(\bigraph{bwd1v1v1p1v1x0p1x0p0x1duals1v1v3x2x1}, \bigraph{bwd1v1v1p1v1x0p1x0p0x1duals1v1v3x2x1}\right), \displaybreak[1]\notag\\
    & \left(\bigraph{bwd1v1v1p1v1x1duals1v1v1}, \bigraph{bwd1v1v1p1v1x1duals1v1v1}\right), \displaybreak[1]\notag\\
    & \left(\bigraph{bwd1v1v1p1v1x0p1x0p0x1p0x1duals1v1v4x3x2x1}, \bigraph{bwd1v1v1p1v1x0p1x0p0x1p0x1duals1v1v4x3x2x1}\right), \displaybreak[1]\notag\\
    & \left(\bigraph{bwd1v1v1p1v1x0p1x0p0x1p0x1duals1v1v2x1x4x3}, \bigraph{bwd1v1v1p1v1x0p1x0p0x1p0x1duals1v1v2x1x4x3}\right), \displaybreak[1]\notag\\
    & \left(\bigraph{bwd1v1v1p1v1x0p1x0p0x1p0x1duals1v1v2x1x4x3}, \bigraph{bwd1v1v1p1v1x0p1x0p0x1p0x1duals1v1v2x1x3x4}\right), \displaybreak[1]\notag\\
    & \left(\bigraph{bwd1v1v1p1v1x0p1x0p0x1p0x1duals1v1v2x1x4x3}, \bigraph{bwd1v1v1p1v1x0p1x0p0x1p0x1duals1v1v1x2x3x4}\right), \displaybreak[1]\notag\\
    & \left(\bigraph{bwd1v1v1p1v1x0p1x0p0x1p0x1duals1v1v2x1x3x4}, \bigraph{bwd1v1v1p1v1x0p1x0p0x1p0x1duals1v1v2x1x3x4}\right), \displaybreak[1]\notag\\
    & \left(\bigraph{bwd1v1v1p1v1x0p1x0p0x1p0x1duals1v1v2x1x3x4}, \bigraph{bwd1v1v1p1v1x0p1x0p0x1p0x1duals1v1v1x2x4x3}\right), \displaybreak[1]\notag\\
    & \left(\bigraph{bwd1v1v1p1v1x0p1x0p0x1p0x1duals1v1v2x1x3x4}, \bigraph{bwd1v1v1p1v1x0p1x0p0x1p0x1duals1v1v1x2x3x4}\right), \displaybreak[1]\notag\\
    & \left(\bigraph{bwd1v1v1p1v1x0p1x0p0x1p0x1duals1v1v1x2x3x4}, \bigraph{bwd1v1v1p1v1x0p1x0p0x1p0x1duals1v1v1x2x3x4}\right), \displaybreak[1]\notag\\
    & \left(\bigraph{bwd1v1v1p1v1x0p1x0p1x0duals1v1v3x2x1}, \bigraph{bwd1v1v1p1v1x0p1x0p1x0duals1v1v3x2x1}\right), \displaybreak[1]\notag\\
    & \left(\bigraph{bwd1v1v1p1v1x0p1x0p1x0duals1v1v3x2x1}, \bigraph{bwd1v1v1p1v1x0p1x0p1x0duals1v1v1x2x3}\right), \displaybreak[1]\notag\\
    & \left(\bigraph{bwd1v1v1p1v1x0p1x0p1x0duals1v1v1x2x3}, \bigraph{bwd1v1v1p1v1x0p1x0p1x0duals1v1v1x2x3}\right), \displaybreak[1]\notag\\
    & \left(\bigraph{bwd1v1v1p1v1x0p1x0p0x1duals1v1v2x1x3}, \bigraph{bwd1v1v1p1v1x0p1x0p0x1duals1v1v2x1x3}\right), \displaybreak[1]\notag\\
    & \left(\bigraph{bwd1v1v1p1v1x0p1x0p0x1duals1v1v2x1x3}, \bigraph{bwd1v1v1p1v1x0p1x0p0x1duals1v1v1x2x3}\right), \displaybreak[1]\notag\\
    & \left(\bigraph{bwd1v1v1p1v1x0p1x0p0x1duals1v1v1x2x3}, \bigraph{bwd1v1v1p1v1x0p1x0p0x1duals1v1v1x2x3}\right), \displaybreak[1]\notag\\
    & \left(\bigraph{bwd1v1v1p1v1x0p1x0duals1v1v2x1}, \bigraph{bwd1v1v1p1v1x0p1x0duals1v1v2x1}\right), \displaybreak[1]\notag\\
    & \left(\bigraph{bwd1v1v1p1v1x0p1x0duals1v1v2x1}, \bigraph{bwd1v1v1p1v1x0p1x0duals1v1v1x2}\right), \displaybreak[1]\notag\\
    & \left(\bigraph{bwd1v1v1p1v1x0p1x0duals1v1v1x2}, \bigraph{bwd1v1v1p1v1x0p1x0duals1v1v1x2}\right), \displaybreak[1]\notag\\
    & \left(\bigraph{bwd1v1v1p1v1x0p0x1duals1v1v2x1}, \bigraph{bwd1v1v1p1v1x0p0x1duals1v1v2x1}\right), \displaybreak[1]\notag\\
    & \left(\bigraph{bwd1v1v1p1v1x0p0x1duals1v1v1x2}, \bigraph{bwd1v1v1p1v1x0p0x1duals1v1v1x2}\right), \displaybreak[1]\notag\\
    & \left(\bigraph{bwd1v1v1p1v1x0duals1v1v1}, \bigraph{bwd1v1v1p1v1x0duals1v1v1}\right)\displaybreak[1]\notag
\Big\} \Big)
\end{align}
\end{lem}

We now apply the triple point obstruction to rule out many of these weeds.

\begin{lem}
There are no subfactors with principal graph pairs starting like any of the last 6 weeds in Equation \eqref{eq:o1a}, any of the last 2 weeds in Equation \eqref{eq:o1c} or any of the last 20 weeds in Equation \eqref{eq:e1}.
\end{lem}
\begin{proof}
This is an immediate consequence of Corollary \ref{lem:odd-triple-point-local} (for the weeds in Equations \eqref{eq:o1a} and \eqref{eq:o1c}) and of Corollary \ref{lem:even-triple-point-local} (for the weeds in Equation \eqref{eq:e1}). The standard example is the second last pair from Equation \eqref{eq:o1a}, which was treated by Haagerup. In our notation, $A_1 = A_1'$ and $A_2 = A_2'$, so the pair is forbidden. 
\end{proof}

At this point, it is convenient to partition the remaining weeds which begin with a triple point  according to their annular multiplicities. Thus we define
\begin{align*}
\cW_{12}  = \Big\{ & \left(\bigraph{bwd1v1v1v1p1v1x1duals1v1v1x2}, \bigraph{bwd1v1v1v1p1v1x1duals1v1v1x2}\right), \displaybreak[1]\\
    & \left(\bigraph{bwd1v1v1v1p1v1x0p1x0p0x1p0x1duals1v1v1x2}, \bigraph{bwd1v1v1v1p1v1x0p0x1p1x0p0x1duals1v1v1x2}\right), \displaybreak[1]\\
    & \left(\bigraph{bwd1v1v1v1p1v1x1duals1v1v2x1}, \bigraph{bwd1v1v1v1p1v1x1duals1v1v2x1}\right), \displaybreak[1]\\
    & \left(\bigraph{bwd1v1v1v1p1v1x0p1x0p0x1p0x1duals1v1v2x1}, \bigraph{bwd1v1v1v1p1v1x0p0x1p1x0p0x1duals1v1v2x1}\right), \displaybreak[1]\\
    & \left(\bigraph{bwd1v1v1p1v1x1duals1v1v1}, \bigraph{bwd1v1v1p1v1x1duals1v1v1}\right), \displaybreak[1]\\
    & \left(\bigraph{bwd1v1v1p1v1x0p1x0p0x1p0x1duals1v1v4x2x3x1}, \bigraph{bwd1v1v1p1v1x1duals1v1v1}\right), \displaybreak[1]\\
    & \left(\bigraph{bwd1v1v1p1v1x0p1x0p0x1p0x1duals1v1v4x2x3x1}, \bigraph{bwd1v1v1p1v1x0p1x0p0x1p0x1duals1v1v4x2x3x1}\right),
\Big\} \\
\cW_{11}  = \Big\{     & \left(\bigraph{bwd1v1v1v1p1v1x0p1x0p1x0duals1v1v1x2}, \bigraph{bwd1v1v1v1p1v1x0p1x0p0x1duals1v1v1x2}\right), \displaybreak[1]\\
    & \left(\bigraph{bwd1v1v1v1p1v1x0p1x0p0x1duals1v1v1x2}, \bigraph{bwd1v1v1v1p1v1x0p0x1p1x0duals1v1v1x2}\right), \displaybreak[1]\\
    & \left(\bigraph{bwd1v1v1p1v1x0p1x0p0x1duals1v1v3x2x1}, \bigraph{bwd1v1v1p1v1x0p1x0p0x1duals1v1v3x2x1}\right)\Big\}, \\
\intertext{and}
\cW_{10}  = \Big\{     & \left(\bigraph{bwd1v1v1v1p1v1x0p1x0duals1v1v1x2}, \bigraph{bwd1v1v1v1p1v1x0p0x1duals1v1v1x2}\right)\Big\}
\end{align*}
We deal with these sets in the next three subsections.

\subsubsection{Annular multiplicities $12$}
For each of the weeds in $\cW_{12}$, we run the odometer one more step and find that there are no remaining weeds. None of the bigraph pairs in $\cW_{12}$ have a equal or unequal extensions, and only the bigraph pairs
\begin{align*}
\Big(
& \left(\bigraph{bwd1v1v1v1p1v1x1duals1v1v1x2}, \bigraph{bwd1v1v1v1p1v1x1duals1v1v1x2}\right), \\
& \left(\bigraph{bwd1v1v1v1p1v1x1duals1v1v2x1}, \bigraph{bwd1v1v1v1p1v1x1duals1v1v2x1}\right), \\
& \left(\bigraph{bwd1v1v1p1v1x1duals1v1v1}, \bigraph{bwd1v1v1p1v1x1duals1v1v1}\right),  \\
& \left(\bigraph{bwd1v1v1p1v1x0p1x0p0x1p0x1duals1v1v4x2x3x1}, \bigraph{bwd1v1v1p1v1x1duals1v1v1}\right) \Big)
\end{align*}
survive the associativity test, so these are the only vines with annular multiplicities $12$.

\subsubsection{Annular multiplicities $11$}
For each of the weeds in $\cW_{11}$, we can run the odometer, eventually removing all weeds. 
We have the following classification statements:

\begin{align*}
\left( \left(\bigraph{bwd1v1v1v1p1v1x0p1x0p1x0duals1v1v1x2}, \bigraph{bwd1v1v1v1p1v1x0p1x0p0x1duals1v1v1x2}\right), 5, \cV_{11,a}, \eset \right), \\
\left(\left(\bigraph{bwd1v1v1v1p1v1x0p1x0p0x1duals1v1v1x2}, \bigraph{bwd1v1v1v1p1v1x0p0x1p1x0duals1v1v1x2}\right), 5, \cV_{11,b} , \eset\right), \\
\intertext{and}
\left( \left(\bigraph{bwd1v1v1p1v1x0p1x0p0x1duals1v1v3x2x1}, \bigraph{bwd1v1v1p1v1x0p1x0p0x1duals1v1v3x2x1}\right), 5, \cV_{11,c}, \eset \right)
\end{align*}
where
\begin{align*}
\cV_{11,a} & = \Big\{ \left(\bigraph{bwd1v1v1v1p1v1x0p1x0p0x1v1x0x0p0x1x0p0x0x1p0x0x1duals1v1v1x2v4x3x2x1}, \bigraph{bwd1v1v1v1p1v1x0p1x0p1x0duals1v1v1x2}\right) \Big\}, \\
\cV_{11,b} & = \Big\{ \left(\bigraph{bwd1v1v1v1p1v1x0p1x0p0x1v1x0x0p0x0x1v0x1duals1v1v1x2v2x1}, \bigraph{bwd1v1v1v1p1v1x0p0x1p1x0v1x0x0p0x1x0v0x1duals1v1v1x2v2x1}\right), \\
 & \qquad\left(\bigraph{bwd1v1v1v1p1v1x0p1x0p0x1v1x0x0p0x1x0p0x0x1v0x1x0p0x0x1v1x0p0x1v0x1v1duals1v1v1x2v3x2x1v2x1v1}, \bigraph{bwd1v1v1v1p1v1x0p0x1p1x0v1x0x0p0x1x0p0x1x0v1x0x0p0x0x1v1x0v1duals1v1v1x2v3x2x1v1}\right) \Big\}, \\
 \intertext{and}
 \cV_{11,c} & = \Big\{ \left(\bigraph{bwd1v1v1p1v1x0p1x0p0x1v0x1x0p0x0x1v0x1duals1v1v3x2x1v1}, \bigraph{bwd1v1v1p1v1x0p1x0p0x1v0x0x1p0x1x0v1x0duals1v1v3x2x1v1}\right) \Big\}.
 \end{align*}

See Figures \ref{fig:odometer-11,a}, \ref{fig:odometer-11,b} and \ref{fig:odometer-11,c} for the detailed output of the odometer in each of these three cases.

\begin{figure}[!htb]
$$
\begin{tikzpicture}
\tikzset{grow=right,level distance=130pt}
\tikzset{every tree node/.style={draw,fill=white,rectangle,rounded corners,inner sep=2pt}}
\Tree
[.\node{$\!\!\begin{array}{c}\bigraph{bwd1v1v1v1p1v1x0p1x0p1x0duals1v1v1x2}\\\bigraph{bwd1v1v1v1p1v1x0p1x0p0x1duals1v1v1x2}\end{array}\!\!$};]
\end{tikzpicture}
$$
\caption{The odometer, running on $\Gamma_{11,a}$.}
\label{fig:odometer-11,a}
\end{figure}

\begin{figure}[!htb]
$$
\hspace{-2cm}
\scalebox{0.72}{
\begin{tikzpicture}
\tikzset{grow=right,level distance=165pt}
\tikzset{every tree node/.style={draw,fill=white,rectangle,rounded corners,inner sep=2pt}}
\Tree
[.\node{$\!\!\begin{array}{c}\bigraph{bwd1v1v1v1p1v1x0p1x0p0x1duals1v1v1x2}\\\bigraph{bwd1v1v1v1p1v1x0p0x1p1x0duals1v1v1x2}\end{array}\!\!$};
	[.\node{$\!\!\begin{array}{c}\bigraph{bwd1v1v1v1p1v1x0p1x0p0x1v1x0x0p0x0x1duals1v1v1x2v2x1}\\\bigraph{bwd1v1v1v1p1v1x0p0x1p1x0v1x0x0p0x1x0duals1v1v1x2v2x1}\end{array}\!\!$};
		[.\node{$\!\!\begin{array}{c}\bigraph{bwd1v1v1v1p1v1x0p1x0p0x1v1x0x0p0x0x1v0x1duals1v1v1x2v2x1}\\\bigraph{bwd1v1v1v1p1v1x0p0x1p1x0v1x0x0p0x1x0v0x1duals1v1v1x2v2x1}\end{array}\!\!$};]]
	[.\node{$\!\!\begin{array}{c}\bigraph{bwd1v1v1v1p1v1x0p1x0p0x1v1x0x0p0x1x0p0x0x1duals1v1v1x2v3x2x1}\\\bigraph{bwd1v1v1v1p1v1x0p0x1p1x0v1x0x0p0x1x0p0x1x0duals1v1v1x2v3x2x1}\end{array}\!\!$};
		[.\node{$\!\!\begin{array}{c}\bigraph{bwd1v1v1v1p1v1x0p1x0p0x1v1x0x0p0x1x0p0x0x1v0x1x0p0x0x1duals1v1v1x2v3x2x1}\\\bigraph{bwd1v1v1v1p1v1x0p0x1p1x0v1x0x0p0x1x0p0x1x0v1x0x0p0x0x1duals1v1v1x2v3x2x1}\end{array}\!\!$};
			[.\node{$\!\!\begin{array}{c}\bigraph{bwd1v1v1v1p1v1x0p1x0p0x1v1x0x0p0x1x0p0x0x1v0x1x0p0x0x1v1x0p0x1duals1v1v1x2v3x2x1v2x1}\\\bigraph{bwd1v1v1v1p1v1x0p0x1p1x0v1x0x0p0x1x0p0x1x0v1x0x0p0x0x1v1x0duals1v1v1x2v3x2x1v1}\end{array}\!\!$};
				[.\node{$\!\!\begin{array}{c}\bigraph{bwd1v1v1v1p1v1x0p1x0p0x1v1x0x0p0x1x0p0x0x1v0x1x0p0x0x1v1x0p0x1v0x1duals1v1v1x2v3x2x1v2x1}\\\bigraph{bwd1v1v1v1p1v1x0p0x1p1x0v1x0x0p0x1x0p0x1x0v1x0x0p0x0x1v1x0v1duals1v1v1x2v3x2x1v1}\end{array}\!\!$};]]]
		[.\node{$\!\!\begin{array}{c}\bigraph{bwd1v1v1v1p1v1x0p1x0p0x1v1x0x0p0x1x0p0x0x1v0x1x0p0x1x0p0x0x1duals1v1v1x2v3x2x1}\\\bigraph{bwd1v1v1v1p1v1x0p0x1p1x0v1x0x0p0x1x0p0x1x0v1x0x0p0x1x0p0x0x1duals1v1v1x2v3x2x1}\end{array}\!\!$};]]
	[.\node{$\!\!\begin{array}{c}\bigraph{bwd1v1v1v1p1v1x0p1x0p0x1v1x0x0p1x0x0p0x0x1duals1v1v1x2v3x2x1}\\\bigraph{bwd1v1v1v1p1v1x0p0x1p1x0v1x0x0p1x0x0p0x1x0duals1v1v1x2v3x2x1}\end{array}\!\!$};]
	[.\node{$\!\!\begin{array}{c}\bigraph{bwd1v1v1v1p1v1x0p1x0p0x1v1x0x0p0x1x0p0x0x1p0x0x1duals1v1v1x2v4x2x3x1}\\\bigraph{bwd1v1v1v1p1v1x0p0x1p1x0v1x0x0p0x1x0p0x1x0p0x0x1duals1v1v1x2v3x2x1x4}\end{array}\!\!$};]]
\end{tikzpicture}
}
$$
\caption{The odometer, running on $\Gamma_{11,b}$.}
\label{fig:odometer-11,b}
\end{figure}

\begin{figure}[!htb]
$$
\begin{tikzpicture}
\tikzset{grow=right,level distance=130pt}
\tikzset{every tree node/.style={draw,fill=white,rectangle,rounded corners,inner sep=2pt}}
\Tree
[.\node{$\!\!\begin{array}{c}\bigraph{bwd1v1v1p1v1x0p1x0p0x1duals1v1v3x2x1}\\\bigraph{bwd1v1v1p1v1x0p1x0p0x1duals1v1v3x2x1}\end{array}\!\!$};
	[.\node{$\!\!\begin{array}{c}\bigraph{bwd1v1v1p1v1x0p1x0p0x1v0x1x0p0x0x1duals1v1v3x2x1}\\\bigraph{bwd1v1v1p1v1x0p1x0p0x1v0x0x1p0x1x0duals1v1v3x2x1}\end{array}\!\!$};
		[.\node{$\!\!\begin{array}{c}\bigraph{bwd1v1v1p1v1x0p1x0p0x1v0x1x0p0x0x1v0x1duals1v1v3x2x1v1}\\\bigraph{bwd1v1v1p1v1x0p1x0p0x1v0x0x1p0x1x0v1x0duals1v1v3x2x1v1}\end{array}\!\!$};]]
	[.\node{$\!\!\begin{array}{c}\bigraph{bwd1v1v1p1v1x0p1x0p0x1v0x1x0p0x1x0p0x0x1duals1v1v3x2x1}\\\bigraph{bwd1v1v1p1v1x0p1x0p0x1v0x1x0p0x0x1p0x1x0duals1v1v3x2x1}\end{array}\!\!$};
		[.\node{$\!\!\begin{array}{c}\bigraph{bwd1v1v1p1v1x0p1x0p0x1v0x1x0p0x1x0p0x0x1v1x0x0p0x0x1p0x0x1duals1v1v3x2x1v3x2x1}\\\bigraph{bwd1v1v1p1v1x0p1x0p0x1v0x1x0p0x0x1p0x1x0v1x0x0p0x1x0p0x1x0duals1v1v3x2x1v3x2x1}\end{array}\!\!$};]]
	[.\node{$\!\!\begin{array}{c}\bigraph{bwd1v1v1p1v1x0p1x0p0x1v1x0x0p0x1x0p0x0x1p0x0x1duals1v1v3x2x1}\\\bigraph{bwd1v1v1p1v1x0p1x0p0x1v1x0x0p0x0x1p0x1x0p0x0x1duals1v1v3x2x1}\end{array}\!\!$};
		[.\node{$\!\!\begin{array}{c}\bigraph{bwd1v1v1p1v1x0p1x0p0x1v1x0x0p0x1x0p0x0x1p0x0x1v1x0x0x0p0x1x0x0p0x0x1x0duals1v1v3x2x1v2x1x3}\\\bigraph{bwd1v1v1p1v1x0p1x0p0x1v1x0x0p0x0x1p0x1x0p0x0x1v1x0x0x0p0x1x0x0p0x0x1x0duals1v1v3x2x1v3x2x1}\end{array}\!\!$};]]]
\end{tikzpicture}
$$
\caption{The odometer, running on $\Gamma_{11,c}$.}
\label{fig:odometer-11,c}
\end{figure}

\subsubsection{Annular multiplicities $10$}
The single weed in $\cW_{10}$ causes more difficulty. It appears that running the odometer never terminates --- weeds with arbitrarily high depth appear. Thus, we choose a certain set of weeds to terminate at, balancing the desire for a small list of weeds which are as deep as possible with the desire for a small list of vines. The particular choices we've made were influenced by our expectation of various methods eliminating vines or weeds in subsequent papers, in particular the quadratic tangles methods which we will use to eliminate certain weeds with annular multiplicities $10$.

\begin{thm}
There is a classification statement
\begin{equation*}
\left( \left(\bigraph{bwd1v1v1v1p1v1x0p1x0duals1v1v1x2}, \bigraph{bwd1v1v1v1p1v1x0p0x1duals1v1v1x2}\right), 5, \cV_{10}, \cW_{10} \right)
\end{equation*}
where
\begin{align*}
\cV_{10} & = \Big\{ \left(\bigraph{bwd1v1v1v1p1v1x0p0x1v1x0p0x1duals1v1v1x2v2x1}, \bigraph{bwd1v1v1v1p1v1x0p1x0duals1v1v1x2}\right), \displaybreak[1]\\
 & \qquad\left(\bigraph{bwd1v1v1v1p1v1x0p1x0v1x0p0x1duals1v1v1x2v1x2}, \bigraph{bwd1v1v1v1p1v1x0p0x1v1x1duals1v1v1x2v1}\right), \displaybreak[1]\\
 & \qquad\left(\bigraph{bwd1v1v1v1p1v1x0p1x0v1x0p1x0p0x1duals1v1v1x2v1x2x3}, \bigraph{bwd1v1v1v1p1v1x0p0x1v1x1p1x0duals1v1v1x2v1x2}\right), \displaybreak[1]\\
 & \qquad\left(\bigraph{bwd1v1v1v1p1v1x0p1x0v1x0p1x0p0x1duals1v1v1x2v2x1x3}, \bigraph{bwd1v1v1v1p1v1x0p0x1v1x1p1x0duals1v1v1x2v1x2}\right), \displaybreak[1]\\
 & \qquad\left(\bigraph{bwd1v1v1v1p1v1x0p1x0v1x0p1x0p0x1v1x0x0duals1v1v1x2v1x2x3}, \bigraph{bwd1v1v1v1p1v1x0p0x1v1x1p1x0v0x1duals1v1v1x2v1x2}\right), \displaybreak[1]\\
 & \qquad\left(\bigraph{bwd1v1v1v1p1v1x0p0x1v1x1v1v1duals1v1v1x2v1v1}, \bigraph{bwd1v1v1v1p1v1x0p1x0v1x0p0x1v1x1duals1v1v1x2v1x2}\right), \displaybreak[1]\\
 & \qquad\left(\bigraph{bwd1v1v1v1p1v1x0p0x1v1x0p1x0p0x1v0x0x1v1duals1v1v1x2v3x2x1v1}, \bigraph{bwd1v1v1v1p1v1x0p1x0v1x0v1duals1v1v1x2v1}\right), \displaybreak[1]\\
 & \qquad\left(\bigraph{bwd1v1v1v1p1v1x0p0x1v1x0p1x0p0x1v0x1x0p0x1x0p0x0x1v1x0x0p0x1x0p0x0x1p0x0x1p0x0x1duals1v1v1x2v3x2x1v5x4x3x2x1}, \bigraph{bwd1v1v1v1p1v1x0p1x0v1x0v1p1p1duals1v1v1x2v1}\right), \displaybreak[1]\\
 & \qquad\left(\bigraph{bwd1v1v1v1p1v1x0p0x1v1x0p1x0p1x0p0x1v0x0x0x1p0x0x0x1v1x0p0x1duals1v1v1x2v4x2x3x1v1x2}, \bigraph{bwd1v1v1v1p1v1x0p1x0v1x0p1x0v1x0p0x1duals1v1v1x2v1x2}\right), \displaybreak[1]\\
 & \qquad\left(\bigraph{bwd1v1v1v1p1v1x0p0x1v1x0p1x0p1x0p0x1v0x0x0x1p0x0x0x1v1x0p0x1duals1v1v1x2v4x2x3x1v2x1}, \bigraph{bwd1v1v1v1p1v1x0p1x0v1x0p1x0v1x0p0x1duals1v1v1x2v2x1}\right), \displaybreak[1]\\
 & \qquad\left(\bigraph{bwd1v1v1v1p1v1x0p0x1v1x0p1x0p1x0p0x1v0x0x0x1p0x0x0x1v1x0p0x1duals1v1v1x2v4x3x2x1v1x2}, \bigraph{bwd1v1v1v1p1v1x0p1x0v1x0p1x0v1x0p0x1duals1v1v1x2v1x2}\right), \displaybreak[1]\\
 & \qquad\left(\bigraph{bwd1v1v1v1p1v1x0p0x1v1x0p1x0p1x0p0x1v0x0x0x1p0x0x0x1v1x0p0x1duals1v1v1x2v4x3x2x1v2x1}, \bigraph{bwd1v1v1v1p1v1x0p1x0v1x0p1x0v1x0p0x1duals1v1v1x2v2x1}\right), \displaybreak[1]\\
 & \qquad\left(\bigraph{bwd1v1v1v1p1v1x0p1x0v1x0p0x1v1x0p0x1v1x1duals1v1v1x2v1x2v1}, \bigraph{bwd1v1v1v1p1v1x0p0x1v1x0p1x0p0x1p0x1v0x0x0x1p1x0x0x0v1x0p1x0p0x1p0x1duals1v1v1x2v4x2x3x1v1x2x3x4}\right), \displaybreak[1]\\
 & \qquad\left(\bigraph{bwd1v1v1v1p1v1x0p1x0v1x0p0x1v1x0p0x1v1x1duals1v1v1x2v1x2v1}, \bigraph{bwd1v1v1v1p1v1x0p0x1v1x0p1x0p0x1p0x1v0x0x0x1p1x0x0x0v1x0p1x0p0x1p0x1duals1v1v1x2v4x2x3x1v2x1x3x4}\right), \displaybreak[1]\\
 & \qquad\left(\bigraph{bwd1v1v1v1p1v1x0p1x0v1x0p0x1v1x0p0x1v1x1duals1v1v1x2v1x2v1}, \bigraph{bwd1v1v1v1p1v1x0p0x1v1x0p1x0p0x1p0x1v0x0x0x1p1x0x0x0v1x0p1x0p0x1p0x1duals1v1v1x2v4x2x3x1v2x1x4x3}\right), \displaybreak[1]\\
 & \qquad\left(\bigraph{bwd1v1v1v1p1v1x0p1x0v1x0p0x1v1x0p0x1v1x0p1x0p0x1p0x1duals1v1v1x2v1x2v4x3x2x1}, \bigraph{bwd1v1v1v1p1v1x0p0x1v1x0p1x0p0x1p0x1v0x0x0x1p1x0x0x0v1x1duals1v1v1x2v4x2x3x1v1}\right), \displaybreak[1]\\
 & \qquad\left(\bigraph{bwd1v1v1v1p1v1x0p0x1v1x0p1x0p0x1p0x1v0x0x0x1p1x0x0x0v1x1v1v1duals1v1v1x2v4x2x3x1v1v1}, \bigraph{bwd1v1v1v1p1v1x0p1x0v1x0p0x1v1x0p0x1v1x0p1x0p0x1p0x1v1x0x0x1duals1v1v1x2v1x2v4x3x2x1}\right), \displaybreak[1]\\
 & \qquad\left(\bigraph{bwd1v1v1v1p1v1x0p1x0v1x0p0x1v1x0p0x1v1x1v1v1duals1v1v1x2v1x2v1v1}, \bigraph{bwd1v1v1v1p1v1x0p0x1v1x0p1x0p0x1p0x1v0x0x0x1p1x0x0x0v1x0p1x0p0x1p0x1v1x0x1x0duals1v1v1x2v4x2x3x1v1x2x3x4}\right) \Big\}\displaybreak[1]\\
 \intertext{and}
\cW_{10} & = \Big\{ \left(\bigraph{bwd1v1v1v1p1v1x0p1x0v1x0v1p1duals1v1v1x2v1}, \bigraph{bwd1v1v1v1p1v1x0p0x1v1x0p1x0p0x1v0x1x0p0x0x1duals1v1v1x2v3x2x1}\right), \displaybreak[1]\\
 & \qquad\left(\bigraph{bwd1v1v1v1p1v1x0p1x0v1x0p0x1v1x0p1x0p0x1v0x1x0p0x0x1v1x0p0x1p0x1duals1v1v1x2v1x2v2x1}, \bigraph{bwd1v1v1v1p1v1x0p0x1v1x0p1x0p0x1p0x1v0x1x0x0p0x0x0x1p1x0x0x0v1x0x0p0x1x0p0x1x0p0x0x1v0x0x0x1p1x0x0x0p0x1x0x0duals1v1v1x2v4x2x3x1v3x2x1x4}\right), \displaybreak[1]\\
 & \qquad\left(\bigraph{bwd1v1v1v1p1v1x0p1x0v1x0p0x1v1x0p0x1v1x0p0x1v1x0p0x1duals1v1v1x2v1x2v2x1}, \bigraph{bwd1v1v1v1p1v1x0p0x1v1x0p1x0p0x1p0x1v0x0x0x1p1x0x0x0v1x0p0x1v0x1p1x0duals1v1v1x2v4x2x3x1v1x2}\right) \Big\}
\end{align*}
\end{thm}
\begin{proof}
See Figure \ref{fig:odometer-10}.
\end{proof}

\begin{figure}[!htb]
$$
\hspace{-1cm}
\scalebox{0.60}{
\begin{tikzpicture}
\tikzset{grow=right,level distance=170pt}
\tikzset{every tree node/.style={draw,fill=white,rectangle,rounded corners,inner sep=2pt}}
\Tree
[.\node{$\!\!\begin{array}{c}\bigraph{bwd1v1v1v1p1v1x0p1x0duals1v1v1x2}\\\bigraph{bwd1v1v1v1p1v1x0p0x1duals1v1v1x2}\end{array}\!\!$};
	[.\node{$\!\!\begin{array}{c}\bigraph{bwd1v1v1v1p1v1x0p1x0v1x0duals1v1v1x2v1}\\\bigraph{bwd1v1v1v1p1v1x0p0x1v1x0p1x0p0x1duals1v1v1x2v3x2x1}\end{array}\!\!$};
		[.\node{$\!\!\begin{array}{c}\bigraph{bwd1v1v1v1p1v1x0p1x0v1x0v1duals1v1v1x2v1}\\\bigraph{bwd1v1v1v1p1v1x0p0x1v1x0p1x0p0x1v0x0x1duals1v1v1x2v3x2x1}\end{array}\!\!$};]
		[.\node[fill=red!30]{$\!\!\begin{array}{c}\bigraph{bwd1v1v1v1p1v1x0p1x0v1x0v1p1duals1v1v1x2v1}\\\bigraph{bwd1v1v1v1p1v1x0p0x1v1x0p1x0p0x1v0x1x0p0x0x1duals1v1v1x2v3x2x1}\end{array}\!\!$};]
		[.\node{$\!\!\begin{array}{c}\bigraph{bwd1v1v1v1p1v1x0p1x0v1x0v1p1p1duals1v1v1x2v1}\\\bigraph{bwd1v1v1v1p1v1x0p0x1v1x0p1x0p0x1v0x1x0p0x1x0p0x0x1duals1v1v1x2v3x2x1}\end{array}\!\!$};]]
	[.\node{$\!\!\begin{array}{c}\bigraph{bwd1v1v1v1p1v1x0p1x0v1x0p0x1duals1v1v1x2v1x2}\\\bigraph{bwd1v1v1v1p1v1x0p0x1v1x1duals1v1v1x2v1}\end{array}\!\!$};
		[.\node{$\!\!\begin{array}{c}\bigraph{bwd1v1v1v1p1v1x0p1x0v1x0p0x1v1x1duals1v1v1x2v1x2}\\\bigraph{bwd1v1v1v1p1v1x0p0x1v1x1v1duals1v1v1x2v1}\end{array}\!\!$};]]
	[.\node{$\!\!\begin{array}{c}\bigraph{bwd1v1v1v1p1v1x0p1x0v1x0p1x0duals1v1v1x2v1x2}\\\bigraph{bwd1v1v1v1p1v1x0p0x1v1x0p1x0p1x0p0x1duals1v1v1x2v4x2x3x1}\end{array}\!\!$};
		[.\node{$\!\!\begin{array}{c}\bigraph{bwd1v1v1v1p1v1x0p1x0v1x0p1x0v1x0p0x1duals1v1v1x2v1x2}\\\bigraph{bwd1v1v1v1p1v1x0p0x1v1x0p1x0p1x0p0x1v0x0x0x1p0x0x0x1duals1v1v1x2v4x2x3x1}\end{array}\!\!$};]]
	[.\node{$\!\!\begin{array}{c}\bigraph{bwd1v1v1v1p1v1x0p1x0v1x0p1x0duals1v1v1x2v1x2}\\\bigraph{bwd1v1v1v1p1v1x0p0x1v1x0p1x0p1x0p0x1duals1v1v1x2v4x3x2x1}\end{array}\!\!$};
		[.\node{$\!\!\begin{array}{c}\bigraph{bwd1v1v1v1p1v1x0p1x0v1x0p1x0v1x0p0x1duals1v1v1x2v1x2}\\\bigraph{bwd1v1v1v1p1v1x0p0x1v1x0p1x0p1x0p0x1v0x0x0x1p0x0x0x1duals1v1v1x2v4x3x2x1}\end{array}\!\!$};]]
	[.\node{$\!\!\begin{array}{c}\bigraph{bwd1v1v1v1p1v1x0p1x0v1x0p1x0duals1v1v1x2v2x1}\\\bigraph{bwd1v1v1v1p1v1x0p0x1v1x0p1x0p1x0p0x1duals1v1v1x2v4x2x3x1}\end{array}\!\!$};
		[.\node{$\!\!\begin{array}{c}\bigraph{bwd1v1v1v1p1v1x0p1x0v1x0p1x0v1x0p0x1duals1v1v1x2v2x1}\\\bigraph{bwd1v1v1v1p1v1x0p0x1v1x0p1x0p1x0p0x1v0x0x0x1p0x0x0x1duals1v1v1x2v4x2x3x1}\end{array}\!\!$};]]
	[.\node{$\!\!\begin{array}{c}\bigraph{bwd1v1v1v1p1v1x0p1x0v1x0p1x0duals1v1v1x2v2x1}\\\bigraph{bwd1v1v1v1p1v1x0p0x1v1x0p1x0p1x0p0x1duals1v1v1x2v4x3x2x1}\end{array}\!\!$};
		[.\node{$\!\!\begin{array}{c}\bigraph{bwd1v1v1v1p1v1x0p1x0v1x0p1x0v1x0p0x1duals1v1v1x2v2x1}\\\bigraph{bwd1v1v1v1p1v1x0p0x1v1x0p1x0p1x0p0x1v0x0x0x1p0x0x0x1duals1v1v1x2v4x3x2x1}\end{array}\!\!$};]]
	[.\node{$\!\!\begin{array}{c}\bigraph{bwd1v1v1v1p1v1x0p1x0v1x0p0x1duals1v1v1x2v1x2}\\\bigraph{bwd1v1v1v1p1v1x0p0x1v1x0p1x0p0x1p0x1duals1v1v1x2v4x2x3x1}\end{array}\!\!$};
		[.\node{$\!\!\begin{array}{c}\bigraph{bwd1v1v1v1p1v1x0p1x0v1x0p0x1v1x0p0x1duals1v1v1x2v1x2}\\\bigraph{bwd1v1v1v1p1v1x0p0x1v1x0p1x0p0x1p0x1v0x0x0x1p1x0x0x0duals1v1v1x2v4x2x3x1}\end{array}\!\!$};
			[.\node{$\!\!\begin{array}{c}\bigraph{bwd1v1v1v1p1v1x0p1x0v1x0p0x1v1x0p0x1v1x0p0x1duals1v1v1x2v1x2v2x1}\\\bigraph{bwd1v1v1v1p1v1x0p0x1v1x0p1x0p0x1p0x1v0x0x0x1p1x0x0x0v1x0p0x1duals1v1v1x2v4x2x3x1v1x2}\end{array}\!\!$};
				[.\node[fill=red!30]{$\!\!\begin{array}{c}\bigraph{bwd1v1v1v1p1v1x0p1x0v1x0p0x1v1x0p0x1v1x0p0x1v1x0p0x1duals1v1v1x2v1x2v2x1}\\\bigraph{bwd1v1v1v1p1v1x0p0x1v1x0p1x0p0x1p0x1v0x0x0x1p1x0x0x0v1x0p0x1v0x1p1x0duals1v1v1x2v4x2x3x1v1x2}\end{array}\!\!$};]
				[.\node{$\!\!\begin{array}{c}\bigraph{bwd1v1v1v1p1v1x0p1x0v1x0p0x1v1x0p0x1v1x0p0x1v1x1p1x0p0x1duals1v1v1x2v1x2v2x1}\\\bigraph{bwd1v1v1v1p1v1x0p0x1v1x0p1x0p0x1p0x1v0x0x0x1p1x0x0x0v1x0p0x1v1x1p0x1p1x0duals1v1v1x2v4x2x3x1v1x2}\end{array}\!\!$};]]
			[.\node{$\!\!\begin{array}{c}\bigraph{bwd1v1v1v1p1v1x0p1x0v1x0p0x1v1x0p0x1v1x1duals1v1v1x2v1x2v1}\\\bigraph{bwd1v1v1v1p1v1x0p0x1v1x0p1x0p0x1p0x1v0x0x0x1p1x0x0x0v1x0p1x0p0x1p0x1duals1v1v1x2v4x2x3x1v1x2x3x4}\end{array}\!\!$};
				[.\node{$\!\!\begin{array}{c}\bigraph{bwd1v1v1v1p1v1x0p1x0v1x0p0x1v1x0p0x1v1x1v1duals1v1v1x2v1x2v1}\\\bigraph{bwd1v1v1v1p1v1x0p0x1v1x0p1x0p0x1p0x1v0x0x0x1p1x0x0x0v1x0p1x0p0x1p0x1v1x0x1x0duals1v1v1x2v4x2x3x1v1x2x3x4}\end{array}\!\!$};]]
			[.\node{$\!\!\begin{array}{c}\bigraph{bwd1v1v1v1p1v1x0p1x0v1x0p0x1v1x0p0x1v1x1duals1v1v1x2v1x2v1}\\\bigraph{bwd1v1v1v1p1v1x0p0x1v1x0p1x0p0x1p0x1v0x0x0x1p1x0x0x0v1x0p1x0p0x1p0x1duals1v1v1x2v4x2x3x1v2x1x3x4}\end{array}\!\!$};]
			[.\node{$\!\!\begin{array}{c}\bigraph{bwd1v1v1v1p1v1x0p1x0v1x0p0x1v1x0p0x1v1x1duals1v1v1x2v1x2v1}\\\bigraph{bwd1v1v1v1p1v1x0p0x1v1x0p1x0p0x1p0x1v0x0x0x1p1x0x0x0v1x0p1x0p0x1p0x1duals1v1v1x2v4x2x3x1v2x1x4x3}\end{array}\!\!$};]
			[.\node{$\!\!\begin{array}{c}\bigraph{bwd1v1v1v1p1v1x0p1x0v1x0p0x1v1x0p0x1v1x0p1x0p0x1p0x1duals1v1v1x2v1x2v4x3x2x1}\\\bigraph{bwd1v1v1v1p1v1x0p0x1v1x0p1x0p0x1p0x1v0x0x0x1p1x0x0x0v1x1duals1v1v1x2v4x2x3x1v1}\end{array}\!\!$};
				[.\node{$\!\!\begin{array}{c}\bigraph{bwd1v1v1v1p1v1x0p1x0v1x0p0x1v1x0p0x1v1x0p1x0p0x1p0x1v1x0x0x1duals1v1v1x2v1x2v4x3x2x1}\\\bigraph{bwd1v1v1v1p1v1x0p0x1v1x0p1x0p0x1p0x1v0x0x0x1p1x0x0x0v1x1v1duals1v1v1x2v4x2x3x1v1}\end{array}\!\!$};]]]
		[.\node{$\!\!\begin{array}{c}\bigraph{bwd1v1v1v1p1v1x0p1x0v1x0p0x1v1x0p1x0p0x1duals1v1v1x2v1x2}\\\bigraph{bwd1v1v1v1p1v1x0p0x1v1x0p1x0p0x1p0x1v0x1x0x0p0x0x0x1p1x0x0x0duals1v1v1x2v4x2x3x1}\end{array}\!\!$};
			[.\node{$\!\!\begin{array}{c}\bigraph{bwd1v1v1v1p1v1x0p1x0v1x0p0x1v1x0p1x0p0x1v0x1x0p0x0x1duals1v1v1x2v1x2v2x1}\\\bigraph{bwd1v1v1v1p1v1x0p0x1v1x0p1x0p0x1p0x1v0x1x0x0p0x0x0x1p1x0x0x0v1x0x0p0x1x0p0x1x0p0x0x1duals1v1v1x2v4x2x3x1v3x2x1x4}\end{array}\!\!$};
				[.\node[fill=red!30]{$\!\!\begin{array}{c}\bigraph{bwd1v1v1v1p1v1x0p1x0v1x0p0x1v1x0p1x0p0x1v0x1x0p0x0x1v1x0p0x1p0x1duals1v1v1x2v1x2v2x1}\\\bigraph{bwd1v1v1v1p1v1x0p0x1v1x0p1x0p0x1p0x1v0x1x0x0p0x0x0x1p1x0x0x0v1x0x0p0x1x0p0x1x0p0x0x1v0x0x0x1p1x0x0x0p0x1x0x0duals1v1v1x2v4x2x3x1v3x2x1x4}\end{array}\!\!$};]]
			[.\node{$\!\!\begin{array}{c}\bigraph{bwd1v1v1v1p1v1x0p1x0v1x0p0x1v1x0p1x0p0x1v1x0x0p0x1x0p0x0x1duals1v1v1x2v1x2v1x3x2}\\\bigraph{bwd1v1v1v1p1v1x0p0x1v1x0p1x0p0x1p0x1v0x1x0x0p0x0x0x1p1x0x0x0v1x1x0p0x0x1duals1v1v1x2v4x2x3x1v1x2}\end{array}\!\!$};]
			[.\node{$\!\!\begin{array}{c}\bigraph{bwd1v1v1v1p1v1x0p1x0v1x0p0x1v1x0p1x0p0x1v1x0x0p0x1x0p0x0x1duals1v1v1x2v1x2v1x3x2}\\\bigraph{bwd1v1v1v1p1v1x0p0x1v1x0p1x0p0x1p0x1v0x1x0x0p0x0x0x1p1x0x0x0v1x0x0p1x0x0p0x1x0p0x1x0p0x0x1duals1v1v1x2v4x2x3x1v4x2x3x1x5}\end{array}\!\!$};]]
		[.\node{$\!\!\begin{array}{c}\bigraph{bwd1v1v1v1p1v1x0p1x0v1x0p0x1v1x0p1x0p0x1p0x1duals1v1v1x2v1x2}\\\bigraph{bwd1v1v1v1p1v1x0p0x1v1x0p1x0p0x1p0x1v0x1x0x0p0x0x0x1p1x0x0x0p0x0x1x0duals1v1v1x2v4x2x3x1}\end{array}\!\!$};]]
	[.\node{$\!\!\begin{array}{c}\bigraph{bwd1v1v1v1p1v1x0p1x0v1x0p1x0p0x1duals1v1v1x2v1x2x3}\\\bigraph{bwd1v1v1v1p1v1x0p0x1v1x1p1x0duals1v1v1x2v1x2}\end{array}\!\!$};
		[.\node{$\!\!\begin{array}{c}\bigraph{bwd1v1v1v1p1v1x0p1x0v1x0p1x0p0x1v1x0x0duals1v1v1x2v1x2x3}\\\bigraph{bwd1v1v1v1p1v1x0p0x1v1x1p1x0v0x1duals1v1v1x2v1x2}\end{array}\!\!$};]]
	[.\node{$\!\!\begin{array}{c}\bigraph{bwd1v1v1v1p1v1x0p1x0v1x0p1x0p0x1duals1v1v1x2v2x1x3}\\\bigraph{bwd1v1v1v1p1v1x0p0x1v1x1p1x0duals1v1v1x2v1x2}\end{array}\!\!$};]
	[.\node{$\!\!\begin{array}{c}\bigraph{bwd1v1v1v1p1v1x0p1x0v1x0p1x0p0x1duals1v1v1x2v1x2x3}\\\bigraph{bwd1v1v1v1p1v1x0p0x1v1x0p1x0p1x0p0x1p0x1duals1v1v1x2v5x2x3x4x1}\end{array}\!\!$};]
	[.\node{$\!\!\begin{array}{c}\bigraph{bwd1v1v1v1p1v1x0p1x0v1x0p1x0p0x1duals1v1v1x2v1x2x3}\\\bigraph{bwd1v1v1v1p1v1x0p0x1v1x0p1x0p1x0p0x1p0x1duals1v1v1x2v5x3x2x4x1}\end{array}\!\!$};]
	[.\node{$\!\!\begin{array}{c}\bigraph{bwd1v1v1v1p1v1x0p1x0v1x0p1x0p0x1duals1v1v1x2v2x1x3}\\\bigraph{bwd1v1v1v1p1v1x0p0x1v1x0p1x0p1x0p0x1p0x1duals1v1v1x2v5x2x3x4x1}\end{array}\!\!$};]
	[.\node{$\!\!\begin{array}{c}\bigraph{bwd1v1v1v1p1v1x0p1x0v1x0p1x0p0x1duals1v1v1x2v2x1x3}\\\bigraph{bwd1v1v1v1p1v1x0p0x1v1x0p1x0p1x0p0x1p0x1duals1v1v1x2v5x3x2x4x1}\end{array}\!\!$};]]
\end{tikzpicture}
}
$$
\caption{The odometer, running on $\Gamma_{10}$.}
\label{fig:odometer-10}
\end{figure}

\subsection{Quadruple points}
\label{sec:quadruple-points}
In the next subsection we run the odometer on all of the quadruple point weeds. In each case, the odometer runs forever, so we carefully choose a convenient set of stopping points, thus exchanging the current list of weeds for another slightly longer list of more complicated weeds, along with several vines. In \S \ref{sec:killing-quadruple-weeds} we then rule out several of these more complicated weeds by specialized methods, again producing several more vines.

\subsubsection{Running the odometer}

\begin{thm}
There's a classification statement
$$\left( \Gamma_{o2,a}, 5, \cV_{o2,a}, \cW_{o2,a} \right)$$
with
\begin{align*}
\cV_{o2,a} & = \Big\{ \left(\bigraph{bwd1v1v1v1p1p1duals1v1v1x2x3}, \bigraph{bwd1v1v1v1p1p1duals1v1v1x2x3}\right), \displaybreak[1]\\
 & \qquad\left(\bigraph{bwd1v1v1v1p1p1v1x0x0p0x1x0duals1v1v1x2x3}, \bigraph{bwd1v1v1v1p1p1v1x0x0p0x1x0duals1v1v1x2x3}\right) \Big\}\displaybreak[1]\\
 \intertext{and}
\cW_{o2,a} & = \Big\{ \left(\bigraph{bwd1v1v1v1p1p1v1x0x0duals1v1v1x2x3}, \bigraph{bwd1v1v1v1p1p1v1x0x0duals1v1v1x2x3}\right), \displaybreak[1]\\
 & \qquad\left(\bigraph{bwd1v1v1v1p1p1v1x0x0p0x1x0v1x0duals1v1v1x2x3v1}, \bigraph{bwd1v1v1v1p1p1v1x0x0p0x1x0v1x0duals1v1v1x2x3v1}\right) \Big\}
\end{align*}
\end{thm}
\begin{proof}
See Figure \ref{fig:odometer-o2,a}.
\end{proof}
\begin{figure}[!htb]
$$
\begin{tikzpicture}
\tikzset{grow=right,level distance=130pt}
\tikzset{every tree node/.style={draw,fill=white,rectangle,rounded corners,inner sep=2pt}}
\Tree
[.\node{$\!\!\begin{array}{c}\bigraph{bwd1v1v1v1p1p1duals1v1v1x2x3}\\\bigraph{bwd1v1v1v1p1p1duals1v1v1x2x3}\end{array}\!\!$};
	[.\node[fill=red!30]{$\!\!\begin{array}{c}\bigraph{bwd1v1v1v1p1p1v1x0x0duals1v1v1x2x3}\\\bigraph{bwd1v1v1v1p1p1v1x0x0duals1v1v1x2x3}\end{array}\!\!$};]
	[.\node{$\!\!\begin{array}{c}\bigraph{bwd1v1v1v1p1p1v1x0x0p0x1x0duals1v1v1x2x3}\\\bigraph{bwd1v1v1v1p1p1v1x0x0p0x1x0duals1v1v1x2x3}\end{array}\!\!$};
		[.\node[fill=red!30]{$\!\!\begin{array}{c}\bigraph{bwd1v1v1v1p1p1v1x0x0p0x1x0v1x0duals1v1v1x2x3v1}\\\bigraph{bwd1v1v1v1p1p1v1x0x0p0x1x0v1x0duals1v1v1x2x3v1}\end{array}\!\!$};]]]
\end{tikzpicture}
$$
\caption{The odometer, running on $\Gamma_{o2,a}$.}
\label{fig:odometer-o2,a}
\end{figure}

\begin{thm}
There's a classification statement
$$\left( \Gamma_{o2,c}, 5, \cV_{o2,c}, \cW_{o2,c} \right)$$
with
\begin{align*}
\cV_{o2,c} & = \Big\{ \left(\bigraph{bwd1v1v1v1p1p1duals1v1v3x2x1}, \bigraph{bwd1v1v1v1p1p1duals1v1v3x2x1}\right), \displaybreak[1]\\
 & \qquad\left(\bigraph{bwd1v1v1v1p1p1v1x0x0p0x0x1duals1v1v3x2x1}, \bigraph{bwd1v1v1v1p1p1v1x0x0p0x0x1duals1v1v3x2x1}\right) \Big\}\displaybreak[1]\\
 \intertext{and}
\cW_{o2,c} & = \Big\{ \left(\bigraph{bwd1v1v1v1p1p1v1x0x0duals1v1v1x3x2}, \bigraph{bwd1v1v1v1p1p1v1x0x0duals1v1v1x3x2}\right) \Big\}
\end{align*}
\end{thm}
\begin{proof}
See Figure \ref{fig:odometer-o2,c}.
\end{proof}
\begin{figure}[!htb]
$$
\begin{tikzpicture}
\tikzset{grow=right,level distance=130pt}
\tikzset{every tree node/.style={draw,fill=white,rectangle,rounded corners,inner sep=2pt}}
\Tree
[.\node{$\!\!\begin{array}{c}\bigraph{bwd1v1v1v1p1p1duals1v1v3x2x1}\\\bigraph{bwd1v1v1v1p1p1duals1v1v3x2x1}\end{array}\!\!$};
	[.\node[fill=red!30]{$\!\!\begin{array}{c}\bigraph{bwd1v1v1v1p1p1v1x0x0duals1v1v1x3x2}\\\bigraph{bwd1v1v1v1p1p1v1x0x0duals1v1v1x3x2}\end{array}\!\!$};]
	[.\node{$\!\!\begin{array}{c}\bigraph{bwd1v1v1v1p1p1v1x0x0p0x0x1duals1v1v3x2x1}\\\bigraph{bwd1v1v1v1p1p1v0x1x0p0x1x0duals1v1v3x2x1}\end{array}\!\!$};]
	[.\node{$\!\!\begin{array}{c}\bigraph{bwd1v1v1v1p1p1v1x0x0p0x0x1duals1v1v3x2x1}\\\bigraph{bwd1v1v1v1p1p1v1x0x0p0x0x1duals1v1v3x2x1}\end{array}\!\!$};]]
\end{tikzpicture}
$$
\caption{The odometer, running on $\Gamma_{o2,c}$.}
\label{fig:odometer-o2,c}
\end{figure}

\begin{thm}
There's a classification statement
$$\left( \Gamma_{e2}, 5, \cV_{e2}, \cW_{e2} \right)$$
with
\begin{align*}
\cV_{e2} & = \Big\{ \left(\bigraph{bwd1v1v1p1p1duals1v1}, \bigraph{bwd1v1v1p1p1duals1v1}\right), \displaybreak[1]\\
 & \qquad\left(\bigraph{bwd1v1v1p1p1v1x0x0p0x1x0duals1v1v1x2}, \bigraph{bwd1v1v1p1p1v1x0x0p0x1x0duals1v1v1x2}\right), \displaybreak[1]\\
 & \qquad\left(\bigraph{bwd1v1v1p1p1v1x0x0p0x1x0duals1v1v2x1}, \bigraph{bwd1v1v1p1p1v1x0x0p0x1x0duals1v1v2x1}\right), \displaybreak[1]\\
 & \qquad\left(\bigraph{bwd1v1v1p1p1v1x0x0p1x0x0duals1v1v1x2}, \bigraph{bwd1v1v1p1p1v1x0x0p1x0x0duals1v1v1x2}\right), \displaybreak[1]\\
 & \qquad\left(\bigraph{bwd1v1v1p1p1v1x0x0p1x0x0duals1v1v2x1}, \bigraph{bwd1v1v1p1p1v1x0x0p1x0x0duals1v1v1x2}\right), \displaybreak[1]\\
 & \qquad\left(\bigraph{bwd1v1v1p1p1v1x0x0p1x0x0duals1v1v2x1}, \bigraph{bwd1v1v1p1p1v1x0x0p1x0x0duals1v1v2x1}\right), \displaybreak[1]\\
 & \qquad\left(\bigraph{bwd1v1v1p1p1v1x0x0p0x1x0v1x0p0x1duals1v1v2x1}, \bigraph{bwd1v1v1p1p1v1x0x0p0x1x0v1x0p0x1duals1v1v2x1}\right), \displaybreak[1]\\
 & \qquad\left(\bigraph{bwd1v1v1p1p1v1x0x0p0x1x0v1x0p0x1v1x0p0x1duals1v1v2x1v2x1}, \bigraph{bwd1v1v1p1p1v1x0x0p0x1x0v1x0p0x1v1x0p0x1duals1v1v2x1v2x1}\right) \Big\}\displaybreak[1]\\
 \intertext{and}
\cW_{e2} & = \Big\{ \left(\bigraph{bwd1v1v1p1p1v1x0x0duals1v1v1}, \bigraph{bwd1v1v1p1p1v1x0x0duals1v1v1}\right), \displaybreak[1]\\
 & \qquad\left(\bigraph{bwd1v1v1p1p1v1x0x0p0x1x0v1x0duals1v1v1x2}, \bigraph{bwd1v1v1p1p1v1x0x0p0x1x0v1x0duals1v1v1x2}\right), \displaybreak[1]\\
 & \qquad\left(\bigraph{bwd1v1v1p1p1v1x0x0p0x1x0v1x0p0x1duals1v1v1x2}, \bigraph{bwd1v1v1p1p1v1x0x0p0x1x0v1x0p0x1duals1v1v1x2}\right) \Big\}
\end{align*}
\end{thm}
\begin{proof}
See Figure \ref{fig:odometer-e2}.
\end{proof}
\begin{figure}[!htb]
$$
\begin{tikzpicture}
\tikzset{grow=right,level distance=130pt}
\tikzset{every tree node/.style={draw,fill=white,rectangle,rounded corners,inner sep=2pt}}
\Tree
[.\node{$\!\!\begin{array}{c}\bigraph{bwd1v1v1p1p1duals1v1}\\\bigraph{bwd1v1v1p1p1duals1v1}\end{array}\!\!$};
	[.\node[fill=red!30]{$\!\!\begin{array}{c}\bigraph{bwd1v1v1p1p1v1x0x0duals1v1v1}\\\bigraph{bwd1v1v1p1p1v1x0x0duals1v1v1}\end{array}\!\!$};]
	[.\node{$\!\!\begin{array}{c}\bigraph{bwd1v1v1p1p1v1x0x0p1x0x0duals1v1v1x2}\\\bigraph{bwd1v1v1p1p1v1x0x0p1x0x0duals1v1v1x2}\end{array}\!\!$};]
	[.\node{$\!\!\begin{array}{c}\bigraph{bwd1v1v1p1p1v1x0x0p1x0x0duals1v1v2x1}\\\bigraph{bwd1v1v1p1p1v1x0x0p1x0x0duals1v1v1x2}\end{array}\!\!$};]
	[.\node{$\!\!\begin{array}{c}\bigraph{bwd1v1v1p1p1v1x0x0p1x0x0duals1v1v2x1}\\\bigraph{bwd1v1v1p1p1v1x0x0p1x0x0duals1v1v2x1}\end{array}\!\!$};]
	[.\node{$\!\!\begin{array}{c}\bigraph{bwd1v1v1p1p1v1x0x0p0x1x0duals1v1v1x2}\\\bigraph{bwd1v1v1p1p1v1x0x0p0x1x0duals1v1v1x2}\end{array}\!\!$};
		[.\node[fill=red!30]{$\!\!\begin{array}{c}\bigraph{bwd1v1v1p1p1v1x0x0p0x1x0v1x0duals1v1v1x2}\\\bigraph{bwd1v1v1p1p1v1x0x0p0x1x0v1x0duals1v1v1x2}\end{array}\!\!$};]
		[.\node[fill=red!30]{$\!\!\begin{array}{c}\bigraph{bwd1v1v1p1p1v1x0x0p0x1x0v1x0p0x1duals1v1v1x2}\\\bigraph{bwd1v1v1p1p1v1x0x0p0x1x0v1x0p0x1duals1v1v1x2}\end{array}\!\!$};]]
	[.\node{$\!\!\begin{array}{c}\bigraph{bwd1v1v1p1p1v1x0x0p0x1x0duals1v1v2x1}\\\bigraph{bwd1v1v1p1p1v1x0x0p0x1x0duals1v1v2x1}\end{array}\!\!$};
		[.\node{$\!\!\begin{array}{c}\bigraph{bwd1v1v1p1p1v1x0x0p0x1x0v1x0p0x1duals1v1v2x1}\\\bigraph{bwd1v1v1p1p1v1x0x0p0x1x0v1x0p0x1duals1v1v2x1}\end{array}\!\!$};
			[.\node{$\!\!\begin{array}{c}\bigraph{bwd1v1v1p1p1v1x0x0p0x1x0v1x0p0x1v1x0p0x1duals1v1v2x1v2x1}\\\bigraph{bwd1v1v1p1p1v1x0x0p0x1x0v1x0p0x1v1x0p0x1duals1v1v2x1v2x1}\end{array}\!\!$};]]]]
\end{tikzpicture}
$$

\caption{The odometer, running on $\Gamma_{e2}$.}
\label{fig:odometer-e2}
\end{figure}

\subsubsection{Killing quadruple point weeds}
\label{sec:killing-quadruple-weeds}
Finally, we kill some of the remaining weeds, in particular 
$$\left(\bigraph{bwd1v1v1v1p1p1v1x0x0p0x1x0v1x0duals1v1v1x2x3v1}, \bigraph{bwd1v1v1v1p1p1v1x0x0p0x1x0v1x0duals1v1v1x2x3v1}\right)$$ from $\cW_{o2,a}$ and all three weeds from $\cW_{e2}$:
$$\left(\bigraph{bwd1v1v1p1p1v1x0x0duals1v1v1}, \bigraph{bwd1v1v1p1p1v1x0x0duals1v1v1}\right),$$
$$\left(\bigraph{bwd1v1v1p1p1v1x0x0p0x1x0v1x0duals1v1v1x2}, \bigraph{bwd1v1v1p1p1v1x0x0p0x1x0v1x0duals1v1v1x2}\right),$$ $$\left(\bigraph{bwd1v1v1p1p1v1x0x0p0x1x0v1x0p0x1duals1v1v1x2}, \bigraph{bwd1v1v1p1p1v1x0x0p0x1x0v1x0p0x1duals1v1v1x2}\right).$$

The second weed on that list is exactly what is ruled out by the even quadruple point obstruction in Theorem \ref{thm:evenquad}.  The others will take a bit more work.

\begin{thm}
Any subfactor with index below $5$ with principal graph pair starting like $\left(\bigraph{bwd1v1v1v1p1p1v1x0x0p0x1x0v1x0duals1v1v1x2x3v1}, \bigraph{bwd1v1v1v1p1p1v1x0x0p0x1x0v1x0duals1v1v1x2x3v1}\right)$ must be a translate of one of the following graph pairs:
\begin{align*}
\Big\{ & \left(\bigraph{bwd1v1v1v1v1v1p1p1v1x0x0p0x1x0v1x0duals1v1v1v1x2x3v1}, \bigraph{bwd1v1v1v1v1v1p1p1v1x0x0p0x1x0v1x0duals1v1v1v1x2x3v1}\right), \displaybreak[1]\\
 &\left(\bigraph{bwd1v1v1v1v1v1p1p1v1x0x0p0x1x0v1x0v1duals1v1v1v1x2x3v1}, \bigraph{bwd1v1v1v1v1v1p1p1v1x0x0p0x1x0v1x0v1duals1v1v1v1x2x3v1}\right) \Big\}
\end{align*}
\end{thm}
\begin{proof}
We'll do two cases, first without any translation, and then with any non-trivial translation.

Without any translation, the univalent vertex at depth $2$ past the quadruple point has dimension $\frac{[4]}{[3]}$. Now $q$ must lie between $1.59$ and $\frac{1 + \sqrt{5}}{2}$ and we see $1<2 \cos(\frac{\pi}{5}) < \frac{[4]}{[3]} <2 \cos(\frac{\pi}{6}) <2$, which is not an allowed dimension.

Next, we find that we can just run the odometer on $$\Gamma_{5321} = \left(\bigraph{bwd1v1v1v1v1v1p1p1v1x0x0p0x1x0v1x0duals1v1v1v1x2x3v1}, \bigraph{bwd1v1v1v1v1v1p1p1v1x0x0p0x1x0v1x0duals1v1v1v1x2x3v1}\right),$$ and obtain the classification statement
$$\left(\Gamma_{5321} , 5, \cV_{\Gamma_{5321} }, \eset\right)$$
where $\cV_{\Gamma_{5321}}$ is the list of graph pairs appearing in the statement of the theorem.
The output of the odometer appears in Figure \ref{fig:odometer-gamma3}.
\begin{figure}[!htb]
$$
\begin{tikzpicture}
\tikzset{grow=right,level distance=170pt}
\tikzset{every tree node/.style={draw,fill=white,rectangle,rounded corners,inner sep=2pt}}
\Tree
[.\node{$\!\!\begin{array}{c}\bigraph{bwd1v1v1v1v1v1p1p1v0x1x0p0x0x1v0x1duals1v1v1v1x2x3v1}\\\bigraph{bwd1v1v1v1v1v1p1p1v0x0x1p0x1x0v0x1duals1v1v1v1x2x3v1}\end{array}\!\!$};
	[.\node{$\!\!\begin{array}{c}\bigraph{bwd1v1v1v1v1v1p1p1v0x1x0p0x0x1v0x1v1duals1v1v1v1x2x3v1}\\\bigraph{bwd1v1v1v1v1v1p1p1v0x0x1p0x1x0v0x1v1duals1v1v1v1x2x3v1}\end{array}\!\!$};]]
\end{tikzpicture}
$$
\caption{The odometer, running on $\Gamma_{5321} $.}
\label{fig:odometer-gamma3}
\end{figure}
\end{proof}

\begin{thm}
Any subfactor with index below $5$ with principal graph pair starting like $\left(\bigraph{bwd1v1v1p1p1v1x0x0p0x1x0v1x0duals1v1v1x2}, \bigraph{bwd1v1v1p1p1v1x0x0p0x1x0v1x0duals1v1v1x2}\right)$ must be a translate of one of the following graph pairs:
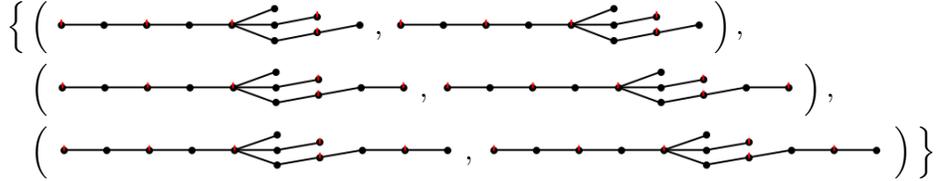
\begin{align*}
 \Big\{ &\left(\bigraph{bwd1v1v1v1v1p1p1v1x0x0p0x1x0v1x0duals1v1v1v1x2}, \bigraph{bwd1v1v1v1v1p1p1v1x0x0p0x1x0v1x0duals1v1v1v1x2}\right), \displaybreak[1]\\
 &\left(\bigraph{bwd1v1v1v1v1p1p1v1x0x0p0x1x0v1x0v1duals1v1v1v1x2v1}, \bigraph{bwd1v1v1v1v1p1p1v1x0x0p0x1x0v1x0v1duals1v1v1v1x2v1}\right), \displaybreak[1]\\
 &\left(\bigraph{bwd1v1v1v1v1p1p1v1x0x0p0x1x0v1x0v1v1duals1v1v1v1x2v1}, \bigraph{bwd1v1v1v1v1p1p1v1x0x0p0x1x0v1x0v1v1duals1v1v1v1x2v1}\right) \Big\}
\end{align*}
\end{thm}
\begin{proof}
First, observe that without any translation, the univalent vertex at depth $3$ has dimension $\frac{[3]}{[2]}$. We know that for any extension of $\Gamma$, $q$ must lie between $1.56$ and $\frac{1 + \sqrt{5}}{2}$. In this range, $1<2 \cos(\frac{\pi}{6}) < \frac{[3]}{[2]} <2 \cos(\frac{\pi}{7}) <2$, which is not an allowed dimension.

Next, we increase the supertransitivity by two and run the odometer on
$\Gamma_{4321} = \left(\bigraph{bwd1v1v1v1v1p1p1v1x0x0p0x1x0v1x0duals1v1v1v1x2}, \bigraph{bwd1v1v1v1v1p1p1v1x0x0p0x1x0v1x0duals1v1v1v1x2}\right),$
but only for a few additional depths, obtaining the classification statement
$$\left(\Gamma_{4321},5,\cV_{\Gamma_{4321}},\{ \Gamma_{4621} \}\right)$$
with $\cV_{\Gamma_{4321}}$ the list of graph pairs appearing in the statement of the theorem and
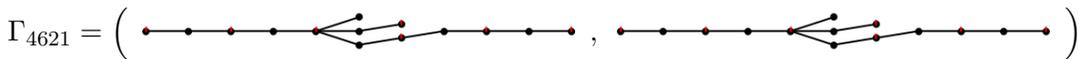
\begin{align*}
\Gamma_{4621} & = \left(\bigraph{bwd1v1v1v1v1p1p1v1x0x0p0x1x0v1x0v1v1v1duals1v1v1v1x2v1v1}, \bigraph{bwd1v1v1v1v1p1p1v1x0x0p0x1x0v1x0v1v1v1duals1v1v1v1x2v1v1}\right)
\end{align*}
The output of the odometer appears in Figure \ref{fig:odometer-gamma2}.
\begin{figure}[!htb]
$$
\scalebox{0.7}{
\begin{tikzpicture}
\tikzset{grow=right,level distance=180pt}
\tikzset{every tree node/.style={draw,fill=white,rectangle,rounded corners,inner sep=2pt}}
\Tree
[.\node{$\!\!\begin{array}{c}\bigraph{bwd1v1v1v1v1p1p1v0x1x0p0x0x1v0x1duals1v1v1v1x2}\\\bigraph{bwd1v1v1v1v1p1p1v0x1x0p0x0x1v0x1duals1v1v1v1x2}\end{array}\!\!$};
	[.\node{$\!\!\begin{array}{c}\bigraph{bwd1v1v1v1v1p1p1v0x1x0p0x0x1v0x1v1duals1v1v1v1x2v1}\\\bigraph{bwd1v1v1v1v1p1p1v0x1x0p0x0x1v0x1v1duals1v1v1v1x2v1}\end{array}\!\!$};
		[.\node{$\!\!\begin{array}{c}\bigraph{bwd1v1v1v1v1p1p1v0x1x0p0x0x1v0x1v1v1duals1v1v1v1x2v1}\\\bigraph{bwd1v1v1v1v1p1p1v0x1x0p0x0x1v0x1v1v1duals1v1v1v1x2v1}\end{array}\!\!$};
			[.\node[fill=red!30]{$\!\!\begin{array}{c}\bigraph{bwd1v1v1v1v1p1p1v0x1x0p0x0x1v0x1v1v1v1duals1v1v1v1x2v1v1}\\\bigraph{bwd1v1v1v1v1p1p1v0x1x0p0x0x1v0x1v1v1v1duals1v1v1v1x2v1v1}\end{array}\!\!$};]]]]
\end{tikzpicture}
}
$$
\caption{The odometer, running on $\Gamma_{4321}$.}
\label{fig:odometer-gamma2}
\end{figure}

For the new weed we've produced, $\Gamma_{4621}$, we again split into cases.

Without any translation, we look at the vertex at depth $10$ and see that it has dimension
$$p(q) = q^{-10} - q^{-8} - 2 q^{-6} - 3 q^{-4} - 4q^{-2} - 6 - 4 q^2 - 3 q^4 - 
 2 q^6 - q^8 + q^{10}.$$
The largest real root of $p$ is at approximately $1.61501$, so for any $q$ in the relevant range, namely $1.6161 < q < \frac{1 + \sqrt{5}}{2}$, $p$ is strictly increasing. Moreover, $p(\frac{1 + \sqrt{5}}{2}) = 1$, so in fact the dimension of this vertex is always strictly less than $1$, which is not allowed.

Translating $\Gamma_{4621}$ by two, we find that the index of the graph is approximately $5.0062 > 5$.
\end{proof}

\begin{thm}
No subfactor with index below $5$ has principal graph which starts like $\bigraph{bwd1v1v1p1p1v1x0x0p0x1x0v1x0p0x1duals1v1v1x2}$.
\end{thm}
\begin{proof}
If you translate this graph by $1$ then its graph norm is $5$ (and in fact the graph is the principal graph for the group-subgroup factor $A_4 \subset A_5$), so we need only consider extensions (rather than translates of extensions).  

We need only consider $q$ between $1.59438$ and $\frac{1+\sqrt{5}}{2}$.  In this range, at least one of the two vertices at depth $5$ will have dimension less than $1$, which is a contradiction.

To see this, notice that the sum of the dimensions of the two vertices at depth $5$ is 
$$\frac{1 - 3 q^4 - 5 q^6 - 3 q^8 + q^{12}}{q^5 + q^7}.$$
This is equal to $2$ near $q=0.61492$ and $q =1.62623$ and smaller than $2$ between those two values.  In particular, it is smaller than $2$ in the range that we are considering.
\end{proof}

\section{Future directions}
\label{sec:future}

As we explained in the introduction, this paper is the first step towards classifying subfactors of index less than $5$.  We complete this classification in a series of subsequent papers.  This project was developed at several Planar Algebra Programming Camps organized by the authors and Emily Peters, and hosted by Vaughan Jones. Further progress was made during a visit by the current authors with Masaki Izumi at Kyoto University. As a result the subsequent papers in the project have a variety of different authors.

All translates of the vines in our classification can all be eliminated (with four exceptions, corresponding to subfactors that actually exist: Haagerup, Asaeda-Haagerup, extended Haagerup, and 2221) by showing that the indices are non-cyclotomic by applying the results of \cite{1004.0665}.  This will be done in a forthcoming paper \cite{index5-part4} by David Penneys and James Tener.

There's not yet a uniform approach to eliminating the weeds, which we instead deal with separately.  In a joint paper with David Penneys and Emily Peters, we will prove that there are no subfactors (of any index) whose principal graphs begin like $\cB$, $\cC$, or $\cF$.  This paper uses several different ``triple point obstructions" coming from the theory of connections and from an identity proved by Vaughan Jones in \cite{quadratic}.  
The quadratic tangles identity does not eliminate the weed $\cB$ as there the rotational eigenvalue is $-1$ and the identity is automatically satisfied.  However, we are able to apply an ad hoc connections argument to remove that case.



In a forthcoming paper with Masaki Izumi and Vaughan Jones we will show that the only subfactor with principal graphs starting like either $\cQ$ or $\cQ'$ is the $3311$ subfactor, which is unique up to taking duals. The quadratic tangles technique from \cite{quadratic} can also be applied to graphs which begin with a quadruple point, and this approach readily rules out subfactors with principal graphs beginning like $\cQ$. A connections argument due to Izumi, followed by a number theoretic argument along the lines of \cite{1004.0665} shows that any principal graph starting like $\cQ'$ must in fact be the $3331$ graph. An involved argument then establishes that any subfactor with this principal graph must be (up to taking dual) the GHJ subfactor.

The final piece of the classification is the uniqueness of the $2221$ subfactor. This has recently been established in the Ph.D. thesis of Richard Han \cite{han-2221}, who was able to derive a set of generators and relations for the corresponding planar algebra directing from the principal graphs.



A natural related question is to consider subfactors of index exactly $5$.  Our techniques generalize easily to understanding possible principal graphs at index equal to $5$, however, the uniqueness problem becomes more difficult.  The techniques for constructing subfactors of integer index are somewhat different from those of non-integer index (in particular, Hopf algebraic techniques, and group cohomology) so we have avoided dealing with the index $5$ case in detail.  However, as was pointed out to us by Izumi, using the results from \cite{MR1491121} simplifies the situation immensely.  We now expect to be able to extend our classification to index equal to $5$.

\newcommand{\urlprefix}{}
\bibliographystyle{alpha}
\bibliography{../../bibliography/bibliography}

This paper is available online at \arxiv{1007.1730}, and at
\url{http://tqft.net/index5-part1}.

\end{document}